\documentclass{ucalgthesnew}
\usepackage{amsfonts}
\usepackage{amsmath}
\usepackage{amsmath}
\usepackage{amssymb}
\usepackage{amsthm}
\usepackage{mathrsfs}
\usepackage{graphicx}
\usepackage[all]{xy}
\usepackage{leftidx}
\usepackage{url}

\newtheorem{theorem}{Theorem}[section]
 \newtheorem{proposition}[theorem]{Proposition}

 \newtheorem{lemma}[theorem]{Lemma}
 \theoremstyle{definition}
 \newtheorem{definition}[theorem]{Definition}
 \newtheorem{example}[theorem]{Example}
 \theoremstyle{remark}
 \newtheorem{remark}[theorem]{Remark}
 \numberwithin{equation}{theorem}

\newcommand{\fg}{{\mathfrak{g}}}

\newcommand{\NN}{{\mathbb{N}}}
\newcommand{\QQ}{{\mathbb{Q}}}
\newcommand{\ZZ}{{\mathbb{Z}}}
\newcommand{\FF}{{\mathbb{F}}}
\newcommand{\RR}{{\mathbb{R}}}
\newcommand{\CC}{{\mathbb{C}}}

\newcommand{\GG}{{\mathbb{G}}}

\newcommand{\oE}{{\lsup{\circ}{\mathcal{E}}}}

\newcommand{\Aut}{{\operatorname{Aut}}}

\newcommand{\kernel}[1]{{\operatorname{ker}\left( #1 \right)}}

\newcommand{\Spec}[1]{{\operatorname{Spec}\left(#1\right)}}

\newcommand{\Hom}{{\operatorname{Hom}}}

\newcommand{\SL}[2]{{\operatorname{SL}(#1)_{#2}}}

\newcommand{\id}{{\operatorname{id}}}

\newcommand{\Gal}[1]{{\operatorname{Gal}(#1)}}

\newcommand{\Fet}[1]{{\operatorname{Fet}(#1)}}

\newcommand{\Fr}{{\operatorname{Fr}}}

\newcommand{\ku}{{K[\mathcal{U}]}}
\newcommand{\conj}{{\operatorname{conj}}}
\newcommand{\pr}{{\operatorname{pr}}}

\newcommand{\invlim}[1]{\lim\limits_{\overleftarrow{#1}}}

\newcommand{\cl}[1]{{\overline{#1}}}  
\newcommand{\un}[1]{{\underline{#1}}} 
\newcommand{\algfg}[1]{{\operatorname{\pi_1}(#1)}}    
\newcommand{\lsup}[2]{{\vphantom{#2}}^{#1}{#2}} 

\newcommand{\tq}{{\ \vert\ }}

\newcommand{\ceq}{{\, :=\, }}
\newcommand{\iso}{{\, \cong\, }}
\newcommand{\ie}{{\emph{i.e.}}}

\newcommand{\Conj}{{\operatorname{conj}}}

\title{Galois Actions on $\ell$-adic Local Systems and Their Nearby Cycles: A Geometrization of Fourier Eigendistributions on the $p$-adic Lie Algebra $\mathfrak{sl}(2)$}
\author{Aaron J Christie}
\year{2012}
\thesis{thesis}
\monthname{September}
\dept{MATHEMATICS \& STATISTICS}
\degree{DOCTOR OF PHILOSOPHY}

\begin{document}

   %
   %
   \makethesistitle

   %
   %

   %
   %
   \altchapter{Abstract}

In this thesis, two $\cl{\QQ}_\ell$-local systems, $\oE$ and $\oE^\prime$ (Definition~\ref{def:thelocalsystems}) on the regular unipotent subvariety $\mathcal{U}_{0,K}$ of $p$-adic $\SL{2}{K}$ are constructed. Making use of the equivalence between $\cl{\QQ}_\ell$-local systems and $\ell$-adic representations of the \'etale fundamental group, we prove that these local systems are equivariant with respect to conjugation by $\SL{2}{K}$ (Proposition~\ref{prop:equivariance}) and that their nearby cycles, when taken with respect to appropriate integral models, descend to local systems on the regular unipotent subvariety of $\SL{2}{k}$, $k$ the residue field of $K$ (Theorem~\ref{prop:mainlocalsystems}). Distributions on $\operatorname{SL}(2,K)$ are then associated to $\oE$ and $\oE^\prime$ (Definition~\ref{def:dists}) and we prove properties of these distributions. Specifically, they are admissible distributions in the sense of Harish-Chandra (Proposition~\ref{prop:properties1}) and, after being transferred to the Lie algebra, are linearly independent eigendistributions of the Fourier transform (Proposition~\ref{prop:properties2}). Together, this gives a geometrization of important admissible invariant distributions on a nonabelian $p$-adic group in the context of the Local Langlands program.

   %
   %
   \altchapter{Acknowledgements}

Somebody pin a medal on Dr. Clifton Cunningham, my supervisor through
two degrees now, or else I'm going to have to make him something with
cardboard and glitter glue, which does him no justice. Dedicated and
enthusiastic to a degree that don't really think I even understand, I
was never at a loss for help or a new idea or just a little shove
whenever I got, let's not say lazy, but...inert. May you end up
with the army of talented grad students you deserve. Now give me the
location of that windmill.

Although it's my general policy to keep the acknowledgements cordoned
off as a jargon-free zone, I can't thank Dr. Patrick Brosnan and all
the participants in the Nearby Cycles Seminar, held between the
Universities of Calgary and British Columbia, without saying nearby
cycles. The results in this thing would not have been possible without the
thorough straightening out my understanding of the nearby cycles
functor received just by sitting and listening.

Moving on from jargon to acronyms, NSERC provided me with money, which
never goes unappreciated, as did the University of Calgary Faculty of
Graduate Studies. All of the drugs I bought with that money were
legal, I promise.

Of course one can't do math and spend money all the time; my friends
and colleagues at the university get the credit for whatever sanity
and proper social adjustment I manage to display. It can't be
easy. Here they are: Karin Arikushi, Amy Cheung, Cameron Hodgins,
Mathias Kaiser, Masoud Kamgar, Max Liprandi, Doug
MacLean, Diane Quan, David Roe, Eric Roettger, Milad Sabeti, Pooyan
Shirvani, Ryan Trelford, Colin Weir, Kjell Wooding.

And of course the most: the moms, Merle Christie, and the sister,
Megan. \emph{Sine qua non.}

   %
   %

   %
   %
   \begin{singlespace}
      \tableofcontents
   \end{singlespace}

   %
   %

   %
   %

   %
   %
   
   \altchapter{Introduction}

Here is what happens in this thesis in a rough nutshell: with a pair
of important admissible invariant distributions on $p$-adic $\operatorname{SL}(2)$ in
mind, a pair of equivariant $\cl{\QQ}_\ell$-local systems on
the \'etale site of $p$-adic $\operatorname{SL}(2)$ are produced. Then it is shown
how these sheaves encode all of the information held in the
distributions. In sum, this provides a new kind of geometrization in
the context of the local Langlands program. There are several crossing
threads in that description; let us start pulling them.

The first is the study of representations of $p$-adic groups, which
can be said to have begun with Mautner's ``\emph{Spherical functions
  over $p$-adic fields}''~\cite{mautnerI} in 1958, and has enjoyed a
remarkably fertile half century since then. Ultimately, the
representations of a $p$-adic group $G$ we want to determine are the
smooth representations, which are those $\pi\colon G\longrightarrow
\Aut(V)$ such that $V = \cup V^K$, where $K$ ranges over all open compact
subgroups of $G$. If each $V^K$ is also finite dimensional, the
representation is called admissible. In that paper and its
sequel~\cite{mautnerII}, Mautner
introduced two important classes of smooth representations,
including supercuspidal representations. All irreducible smooth
representations of a $p$-adic group $G$ are subquotients of
representations induced from supercuspidal representations of
parabolic subgroups of $G$. 

Distributions enter this picture with the work of
Harish-Chandra. Among other things, he developed and introduced the
appropriate notion of a character of an admissible representation as a
distribution on the Hecke algebra of locally constant and compactly
supported functions on $G$~\cite{HC70}. The intimate connection this established
between the representation theory of $p$-adic groups and harmonic
analysis only deepened in the wake of Langlands's work, particularly
the famous conjectural correspondence that bears his name.

A rough (and admittedly nonstandard) way of describing the Local
Langlands Correspondence is to say that it promises a function between
characters of irreducible admissible representations to certain
representations (called L-parameters) of the Galois group of the
local field in play. In general this function is not injective, but it
is predicted that its fibres are parametrized by a specific finite
group. The finite set of representations whose characters appear together in the fibre
above a single parameter is called an L-packet of representations.

The particular distributions at hand in this thesis appear in the
vector space spanned by the characters
of the admissible representations in the unique L-packet of size 4
for $p$-adic $\operatorname{SL}(2)$ and possess other remarkable properties. They are
stably conjugate; they are supported on
topologically unipotent elements; and, when transported to the Lie
algebra, they are linearly independent
eigendistributions---with the same eigenvalue---for the Fourier
transform (in fact, they span the unique
eigenspace for the Fourier transform on the Lie algebra of $p$-adic
$\operatorname{SL}(2)$, but that is not proved here). What is achieved in this
thesis is a geometrization of these distributions: they are
turned into sheaves. Which is the other major thread running
through the background of this thesis.

An example may make this notion of `geometrization' more concrete. Let
$G$ be an algebraic group, commutative and connected, over a finite
field $\FF_q$. The Lang map,
$x\mapsto\frac{\operatorname{Frob}(x)}{x}$, defines an \'etale cover
of $G$ with automorphism group isomorphic $G(\FF_q)$. Therefore, the
\'etale fundamental group $\algfg{G,\cl{g}}$ (described in
Section~\ref{section:fundamentalgroup} of
Chapter~\ref{ch:ladicsheaves}) comes equipped with a map
$\algfg{G,\cl{g}}\rightarrow G(\FF_q)$. Combined with a character $\theta\colon
G(\FF_q)\rightarrow\cl{\QQ}_\ell^\times$ and the canonical map
$\algfg{\cl{G},\cl{g}}\rightarrow\algfg{G,\cl{g}}$,
where $\cl{G} = G\times_{\Spec{\FF_q}}\Spec{\cl{\FF}_q}$, we get a
character of $\algfg{\cl{G},\cl{g}}$:
\[
\algfg{\cl{G},\cl{g}}\rightarrow\algfg{G,\cl{g}}\rightarrow
G(\FF_q)\stackrel{\theta}{\rightarrow}\cl{\QQ}_\ell^\times .
\]
This character is equivalent to an $\ell$-adic local system
$\mathcal{L}$ on $\cl{G}$ possessing an isomorphism
\[
\phi\colon\operatorname{Frob}^*\mathcal{L}\stackrel{\sim}{\longrightarrow}\mathcal{L}.
\]
For any $\FF_q$-point $\cl{g}$ of $G$, the composition of the
canonical isomorphism 
\[
\mathcal{L}_{\cl{g}} =
\mathcal{L}_{\operatorname{Frob}(\cl{g})} \iso
(\operatorname{Frob}^*\mathcal{L})_{\cl{g}}
\]
with the map on stalks
induced by $\phi$ gives an automorphism of $\mathcal{L}_{\cl{g}}$
whose trace belongs to $\cl{\QQ}_\ell^\times$. The function
$t^{\mathcal{L}}_{\operatorname{Frob}}\colon
G(\FF_q)\rightarrow\cl{\QQ}_\ell^\times$ this defines is equal to $\theta$.

The local system $\mathcal{L}$ is in fact an example of a character sheaf, which is the
foremost example of geometrization in representation theory. George
Lusztig introduced character sheaves in a
series of papers that began to appear in 1985~\cite{CShI, CShII,
  CShIV, CShV, CShVprime}. In their original iteration, character
sheaves were defined on general linear groups over an algebraically
closed field of positive characteristic. From them, Lusztig was able
to obtain, by means of the Grothendieck-Lefschetz trace formula, all
irreducible representations of $\operatorname{GL}(n,\FF_{p^m})$ for
all $n$ and $m$, uniting the representation theory of all these groups in an
object related to the geometry of $\operatorname{GL}(n)$.

Lusztig has gone on to generalize character sheaves to larger and
larger classes of groups~\cite{LusCSandGen}, and despite the fact that they are only ever
applied to groups over algebraically closed fields of positive
characteristic, their definition makes sense even for groups over
fields of characteristic 0. Further expansions of the character sheaf
perspective have been carried out by, for example, Drinfeld and his
collaborators (in positive characteristic), and by Cunningham and
Salmasian, who have taken steps to develop character sheaves for use
in characteristic 0~\cite{cunningsalmasian}. 

As one might expect, this latter effort meets extra difficulty. For
example, if $G$ is a linear algebraic group defined over the generic
fibre of a Henselian trait, accessing Lusztig's methods---or more to
the point, accessing the trace formula---requires moving whatever
geometric objects one decides to work with from $G$ to the special
fibre of an integral model for $G$ (or some
scheme closely related to $G$). A choice of integral model for $G$ makes this
possible. And, fortunately, after making such a choice there is machinery suited to
moving sheaves from the generic fibre to the special fibre: Deligne's adaptation of the nearby cycles
functor to schemes defined over Henselian traits.
As described in SGA 7(II)~\cite{SGA7ii}, the nearby cycles functor is robust
enough to even allow the addition of Galois actions to the sheaves one
chooses to study, making it possible to work even on groups defined
over algbraically unclosed fields. Yet another degree of freedom is
offered by the multiplicity of integral models that can be chosen
to use in the definition of the functor.

With all this, what remains is finding appropriate geometric objects
(by which we always mean sheaves of some kind or other) on $G$ so
that, when fed into this machinery, it reproduces aspects of the
representation theory of $p$-adic groups as it is known to us
presently. One use of this perspective is to try to replace the basic
objects on both sides of the Local Langlands Correspondence with a
geometric avatar, which would allow the techniques of algebraic geometry to be
brought to bear on the correspondence itself. Recent work (in
preparation) by Achar, Cunningham, Kamgarpour, and Salmasian has found
how to geometrize the Galois side of the Local Langlands
Correspondence. This thesis establishes a result in this spirit, which
is what we mean by `geometrization in the context of the local Langlands program.'

Specifically, we work on $G = \SL{2}{K}$, where $K$ is a $p$-adic
field with finite residue field $k$ of positive characteristic. In
Chapter~\ref{ch:somelocalsystems}, we define two local systems on the
regular unipotent subvariety $\mathcal{U}_{0,K}$, which are
our chosen geometric objects on $G$, and equip their base change to
$\mathcal{U}_{0,\cl{K}}$ with an action of $\Gal{\cl{K}/K}$. In fact,
these local systems are very special: when extended by zero from
$\mathcal{U}_{0,K}$, they define two perverse sheaves on $G$ that are
identical after change of base to $\cl{K}$. That perverse sheaf on
$\cl{G}$ is the unique cuspidal unipotent character sheaf for
$\cl{G}$. Deploying the heavy machinery of perverse sheaves and
character sheaves is not necessary to prove the results in this
thesis, so these facts are not proved here. But the results we do
obtain show that the unique cuspidal unipotent
character sheaf for $\cl{G}$ can be equipped with two very different
Galois actions, and shows how the resulting Galois sheaves determine
two very different, and very special, distributions on $G(K)$. 

These result are contained in the remainder of
Chapter~\ref{ch:somelocalsystems} and in Chapter~\ref{ch:mainresult},
where important properties of the pair of local systems and their
nearby cycles are shown. In particular, it is shown that the local systems are equivariant with
respect to the conjugation action of $G$ on itself
(Proposition~\ref{prop:equivariance}), and that with respect to
certain integral models for $\mathcal{U}_{0,K}$ their nearby
cycles---which are local systems on $\mathcal{U}_{0,\cl{k}}$ equipped with an action of
$\Gal{\cl{K}/K}$---descend to local systems on
$\mathcal{U}_{0,k}$ (Theorem~\ref{prop:mainlocalsystems}). 
Finally, in Chapter~\ref{ch:distributions}, we apply all those results by defining, with
the aid of the Grothendieck-Lefschetz trace formula, distributions on
the Hecke algebra of $G(K)$ (and $\mathfrak{g}(K)$, $\mathfrak{g}$ the
Lie algebra of $G$). These distributions are shown to be linearly independent invariant
admissible distributions that are eigendistributions of the Fourier
transform (Propositions~\ref{prop:properties1}
and~\ref{prop:properties2}). Together, this establishes a successful
geometrization.

Before all that, the first two chapters are spent providing background for the work of
Chapter~\ref{ch:somelocalsystems}. Chapter~\ref{ch:ladicsheaves} gives
the definitions and basic properties of $\pi$-adic sheaves and
$\cl{\QQ}_\ell$-local systems, which are the basic categories
underlying the geometric objects at the centre of the thesis. Despite
the name, these are not honest sheaves, but projective systems
composed of sheaves. As such, they are somewhat cumbersome, but, as
was only hinted at in the example above, there is an equivalence of categories
between $\cl{\QQ}_\ell$-local systems and $\ell$-adic representations
of the \'etale fundamental group. Accordingly, the \'etale fundamental
group is defined and this equivalence is described in
Chapter~\ref{ch:ladicsheaves}. Chapter~\ref{ch:galoisactions}
recalls Deligne's study of Galois actions on sheaves in
SGA~\cite[Expos\'e XIII]{SGA7ii} and describes its extension to
$\cl{\QQ}_\ell$-local systems, which is a crucal ingredient of the
results in Chapter~\ref{ch:mainresult}. The opening sections of
Chapter~\ref{ch:mainresult} recall the other crucial ingredient,
the nearby cycles functor.

Altogether, this thesis represents the first example of
geometrization, in the context of the local Langlands program, of
sophisticated admissible invariant distributions on nonabelian
$p$-adic groups. As such, it invites a question: can this be done
more generally? Already, joint work with Aubert and Cunningham (in
preparation) has found that the techniques employed in this thesis can
indeed be used to geometrize other distributions on other $p$-adic
groups. However, those distributions are also eigendistributions of
the Fourier transform and, as such, represent a relatively small
class. And, the local systems that appear as geometric avatars of
distributions in that class are, like our local systems, closely
related to cuspidal unipotent character sheaves. So, is it possible to
geometrize all admissible invariant distributions on $p$-adic groups?
While the answer is not known, and
certainly depends on what is meant by geometrization, the work here
points a way forward: by combining the techniques developed in this
thesis with recent work by Cunningham and Roe on the geometrization of
characters of $p$-adic tori, it appears that it may be possible to
find geometric avatars for a significant class of admissible invariant
distributions, and in doing so, shed light on the Local Langlands
Correspondence, one of the most important open problems in number theory.



   %
   %
   \mainmatter

\chapter{The \'Etale Fundamental Group, $\pi$-adic Sheaves,
  $\cl{\QQ}_\ell$-Local Systems}\label{ch:ladicsheaves} 

The main results of this thesis concern a pair of
$\cl{\QQ}_\ell$-local systems defined on the smooth subvariety of
regular unipotent elements in $\SL{2}{K}$. The category of
$\cl{\QQ}_\ell$-local systems is obtained from the category of $\pi$-adic
sheaves. These are somewhat complicated and unwieldy objects, so it is
our preference to exploit an equivalence between
$\cl{\QQ}_\ell$-local systems and the category of $\ell$-adic
representations of the \'etale fundamental group, and work instead
with the representations rather than the local systems themselves.
Therefore, this chapter contains information on the \'etale
fundamental group, $\pi$-adic sheaves, $\cl{\QQ}_\ell$-local systems
tailored toward understanding the equivalence of categories that will be used throughout the
rest of thesis. Because all the results in this chapter are standard,
we avoid giving proofs except when such details help with understanding some
aspect of the work in later chapters, but do provide references to works where complete proofs
can be found.

\section{\'Etale Fundamental Group}\label{section:fundamentalgroup}

The theory of the \'etale fundamental group originated with
Grothendieck, and the definitive reference for all the material
in this section are the volumes of SGA (in particular SGA
1~\cite{SGA1}). Another reference, leaner and better suited to the
needs of this thesis, is Szamuely's \emph{Galois Groups and Fundamental Groups}~\cite{szamuely}.
\subsection{\'Etale Covers}

A morphism of arbitrary schemes $f\colon Y\rightarrow X$ is \emph{\'etale} if it is of finite type,
flat, and unramified\footnote{Along with a basic familiarity with
  algebraic geometry, we elect to assume the reader's
  familiarity with these terms. In any event, they are fully explained in
  the sources listed above.}. An \emph{\'etale cover} is an \'etale morphism that is also finite
and surjective. The category of \'etale covers of a scheme $X$ will be denoted
$\Fet{X}$. 

Some basic properties enjoyed by \'etale morphisms and \'etale covers, all but
immediate from the definition (the reader may also consult
Milne~\cite[Ch.~1]{milnebook}), are:
\begin{itemize}
\item The composite of two \'etale morphisms (resp. covers) is
  also an \'etale morphism (resp. cover).
\item If $g\circ f$ and $g$ are \'etale morphisms (resp. covers),
  then $f$ is also an \'etale morphism (resp. cover).
\item If $f\colon Y\rightarrow X$ is \'etale, and $g\colon
  Z\rightarrow X$ is any morphism, the map $f^Z\colon Y\times_X Z\rightarrow Z$
  obtained by pullback (the \emph{base change} of $f$ by $g$) is again
  \'etale. The same is true for \'etale covers. 
\end{itemize}
\begin{example}\label{ex:etalecovers}
\begin{enumerate}
\item Some examples are immediately obvious:  any scheme
is an \'etale cover of itself, as is the disjoint union of finitely many
copies of that scheme. \'Etale covers of this form are called
\emph{trivial}. A scheme with no nontrivial \'etale covers is
\emph{simply connected}.
\item\label{eg:etaleoverspeck} \'Etale covers of the spectrum of a field $K$ are all by
  spectra of finite products of finite separable extensions of $K$
  (which are quite sensibly called finite \'etale $K$-algebras).
\item\label{eg:etaleoverGm} The connected \'etale covers of the multiplicative group scheme
  \[\GG_{m,K} = \Spec{K[x,y]/(xy-1)}\] when $K$ is algebraically closed
    are, after passing to coordinate rings, of the
    form \begin{eqnarray*} [d]\colon K[x,y]/(xy-1) & \longrightarrow
      & K[x,y]/(xy-1) \\ x & \longmapsto & x^d \\ y & \longmapsto &
      y^{-d} \end{eqnarray*} for any $d\in\ZZ$.
  \par When $K$ is not algebraically closed, the
    \'etale covers are \begin{eqnarray*} [d]_F\colon K[x,y]/(xy-1) & \longrightarrow
      & F[x,y]/(xy-1) \\ K & \hookrightarrow& F \\ x & \longmapsto & x^d \\ y & \longmapsto &
      y^{-d} \end{eqnarray*} for any integer $d$ and finite separable
    extension $K\hookrightarrow F$.
\item\label{eg:etaleoveraffineline} More examples emerge in the wake of the following fact: for an
integral normal scheme $X$ with function field $F$, every connected
\'etale cover of $X$ arises as the
normalization of $X$ in a finite separable exension of
$F$.\footnote{See Lenstra~\cite[Thm.~6.13]{lenstra} for a proof.} But not every
finite separable extension of $K$ yields an \'etale cover; the
resulting morphism may be ramified. This fact combined with
Minkowski's theorem shows that there are no nontrivial \'etale covers
of $\Spec{\ZZ}$, nor are there any of $\mathbb{A}^1_{K}$ when $K$ is
algebraically closed, which accords with the simple connectedness of
the affine line over $\CC$.
\end{enumerate}
\end{example}

 Beyond the properties listed,
\'etale covers possess, in parallel with ordinary
topological covering spaces, a local triviality
property. Namely, if $f\colon Y\rightarrow X$ is an \'etale cover and
$X$ is connected, then there exists an \'etale cover $g\colon
Z\rightarrow X$ such that $Y\times_XZ$ is isomorphic to a disjoint
union of finitely many copies of $Z$.\footnote{This is proved in
  Szamuely~\cite[Prop. 5.2.9]{szamuely}.} This fact is the foundation
that the categorical equivalence between local systems and fundamental
group representations is built upon.

\subsection{Finite Locally Constant Sheaves}

A sheaf $\mathcal{F}$ on the \'etale site of a scheme $X$ is \emph{finite locally constant} if
there exists a covering by \'etale opens\footnote{Meaning that, taken together, the images of the maps $X_i
  \rightarrow X$ are surjective onto the underlying topological space
  of $X$.} $\{X_i\rightarrow X\tq i\in
I\}$ such that $\mathcal{F}|_{X_i}$ is constant with finite
stalks. This is equivalent to the set of sections of $\mathcal{F}|_{X_i}$ over
any \'etale open being finite. We will write $\operatorname{Flc}(X)$
for the category of finite locally constant
sheaves on $X$.

There is functor from $\operatorname{Fet}(X)$ to
$\operatorname{Flc}(X)$, given by associating to every \'etale cover
$f\colon Y\rightarrow X$ its sheaf of sections, $\mathcal{F}_Y$. For each \'etale
scheme $U$ over $X$,
\[
\mathcal{F}_Y(U)\ceq \{\sigma\colon U\rightarrow U\times_XY\tq
f^U\circ\sigma = \id_U\}.
\]
If we restrict attention to each connected component of $X$, that
$\mathcal{F}_Y$ is a sheaf of \emph{finite} sets is just a consequence
of the fact that $f$ is a finite morphism; local constancy is a result of the
local triviality of $f$ together with the fact that any section of an \'etale
cover induces an isomorphism of connected components of the base space onto
connencted components of the cover\footnote{A proof of this can be
  found in~\cite[Prop. 5.3.1]{szamuely}.}. Descent techniques whose
description would lead too far afield
provide a quasi-inverse to this functor, proving an equivalence of
categories that we record here for reference later\footnote{This is
  proved in Conrad's notes~\cite[Thm. 1.1.7.2]{conradetale}.}:
\begin{proposition}\label{prop:coverlocallyconstantcorr}
The category $\Fet{X}$ of a scheme $X$ is
equivalent to the category $\operatorname{Flc}(X)$ of finite locally
constant sheaves of sets on $X$. 
\end{proposition} 

With this equivalence in hand, the connection between finite locally constant
 sheaves and the \'etale fundamental group is a
relatively simple matter to establish because the \'etale fundamental group
classifies \'etale covers, just as the topological fundamental group
classifies topological covers. Therefore, we turn to describing the
\'etale fundamental group.

\subsection{Geometric Points, Fibres, and Galois Covers}

In parallel with the topological
case, defining the \'etale fundamental group requires endowing both
the base space and its covering spaces with base points. A
\emph{geometric point} $\cl{y}$ of a scheme $Y$ is a morphism
\[
\cl{y}\colon \Spec{\Omega}\longrightarrow Y,
\]
where $\Omega$ is a separably closed field. This is equivalent to a
choice of point $y$ in the underlying topological space of
$Y$ together with an embedding of the residue field at $y$ into a
separably closed field. If $f\colon Y\rightarrow X$ is a morphism of schemes, we write
$f(\cl{y})$ for the composition $f\circ\cl{y}$ --- the image of
$\cl{y}$ in $X$.

Now, if $\cl{x}\colon \Spec{\Omega}\rightarrow X$ is a geometric point of
$X$ and $f\colon Y\rightarrow X$ is an \'etale cover, the
\emph{geometric fibre} of $Y$ above $\cl{x}$ is the set
\[
F_\cl{x}(Y) \ceq \{\cl{y}\colon\Spec{\Omega}\rightarrow Y\tq f(\cl{y})= \cl{x}\}.
\]
\begin{example}\label{ex:geometricfibres}
\begin{enumerate}
\item A choice of geometric point for $\Spec{K}$ amounts to a choice
  of embedding of $K$ into the separably closed field $\Omega$. The
  \'etale covers of $\Spec{K}$ are given by inclusions of $K$ into
  finite \'etale $K$-algebras. The elements of the geometric fibre of
  any such cover is therefore the collection of embeddings of its component
  finite separable extensions into $\Omega$ that are consistent with the choice of
  embeddings of $K$ into $\Omega$ and the algebra.
\item A geometric point of $\GG_{m,K} = \Spec{K[x,y]/(xy-1)}$ is given
    by a choice of image for $x$ in $\Omega^\times$. With such a
    choice, the geometric fibre of a connected \'etale cover (as
    described in Example~\ref{ex:etalecovers}.\ref{eg:etaleoverGm})
    consists, when $K$ is algebraically closed, of points $\cl{y}$
    such that
    \[ 
\xymatrix{ K[x,y]/(xy-1)\ar[rr]^{\ \ \ \cl{y}} & & \Omega \\
K[x,y]/(xy-1) \ar[u]^{[d]} \ar[urr]_{x\mapsto\omega} & }
    \] commutes. Thus, $\cl{y}$ maps $x$ to a $d^\textrm{th}$ root of
    $\omega$. When $K$ is not algebraically closed, the covering space
    includes a choice of finite separable extension $F$ of $K$, and
    the points of the geometric fibre are a choice of $d^\textrm{th}$
    root of $\omega$ \emph{and} a choice of embedding of $F$ into
    $\Omega$ that respects the embedding of $K$ into both fields.
\end{enumerate}
\end{example}

Being a finite morphism, an \'etale cover necessarily has finite geometric
fibres. When the base scheme is connected, this finite cardinality
must be identical across all geometric points. Even more can be said when the covering
space $X$ is also connected.

For an \'etale cover $f\colon Y\rightarrow X$, write $\Aut_X(Y)$ for
the set of isomorphisms $\varphi$ of $Y$ such that
\[
\xymatrix{
Y \ar@{<->}[rr]^{\varphi} \ar[dr]_f & & Y \ar[dl]^f \\
& X & }
\]
commutes. When the covering map itself requires emphasis, we will use
the alternate notation $\Aut_X(f)$. The automorphism group acts on the geometric fibre of the
cover via
\[
 \varphi\cdot\cl{y} \ceq
\varphi (\cl{y}), 
\] 
and when $Y$ is connected, the fact that any section maps a
connected component of the base onto a connected component of the
cover implies that any $\varphi\in\Aut_X(Y)$ is actually
entirely determined by the image of $\cl{y}$. From that same fact it also
follows that the action of the automorphism group on the fibre has no
fixed points when $Y$ is connected\footnote{These are proved as
  Corollaries 5.3.3 and 5.3.4 in~\cite{szamuely}.}; thus, when
$X$ and $Y$ are both connected, the size of the automorphism group
is bounded by the size of the geometric fibre:
\[
|\Aut_X(Y)|\leq |F_\cl{x}(Y)|.
\]
When $|\Aut_X(Y)| =  |F_\cl{x}(Y)|$, $f$ is said to be a
\emph{Galois} \'etale cover, or simply a Galois cover. 
\begin{example}\label{ex:galoiscovers}
\begin{enumerate}
\item\label{eg:galoisoverSpecK} As in
  Example~\ref{ex:etalecovers}.\ref{eg:etaleoverspeck}, \'etale covers of the spectrum of a field $K$ are all by
  spectra of finite \'etale $K$-algebras. Such
  a cover is Galois when all the constituent extensions are
  Galois, and is connected exactly when it is composed of a single
  finite separable extension.
\item\label{eg:galoisoverGm} All of the connected \'etale covers of
  $\GG_{m,K}$ when $K$ is algebraically closed described in Example~\ref{ex:etalecovers}.\ref{eg:etaleoverGm},
     \begin{eqnarray*} [d]\colon K[x,y]/(xy-1) & \longrightarrow
      & K[x,y]/(xy-1) \\ x & \longmapsto & x^d \\ y & \longmapsto &
      y^{-d}, \end{eqnarray*} for any $d\in\ZZ$, are
    Galois. The automorphism groups of these covers are
\[
\Aut_{\GG_{m,K}}([d]) \iso \mu_{d,K}.
\] 
\par When $K$ is not algebraically closed the
    \'etale covers are \begin{eqnarray*} [d]_F\colon K[x,y]/(xy-1) & \longrightarrow
      & F[x,y]/(xy-1) \\ K & \hookrightarrow& F \\ x & \longmapsto & x^d \\ y & \longmapsto &
      y^{-d} \end{eqnarray*} for any integer $d$ and finite separable
    extension $K\hookrightarrow F$. In this case the covering group is
\[
\Aut_{\GG_{m,K}}([d]_F) \iso \mu_{d,F}\times\Gal{F/K};
\] therefore, the cover is Galois
    precisely when $F$ is a Galois extension containing all $d^\textrm{th}$
    roots of unity. 
\end{enumerate}
\end{example}

\subsection{The Fundamental Group}

Like the fundamental group familiar from topology, the \'etale
fundamental group is meant to classify the \'etale covers in the sense that the
automorphism groups of the covers appear as finite quotients of the
group, with the relationship between quotient groups reflecting a
relationship between the corresponding covers.

To allow this idea to lead us toward the definition, suppose we have an \'etale cover
$f\colon Y\rightarrow X$ and another connected \'etale cover $g\colon
Z\rightarrow X$ such that
\[
\xymatrix{
Y\ar[rr]^{h} \ar[dr]_{f} &  & Z \ar[dl]^{g} \\ & X & }
\]
commutes. Then, as one can find proved
in~\cite[Prop. 5.3.8]{szamuely}, the factorizing map $h$ is also an
\'etale cover exactly when there exists a normal subgroup $H$ of
$\Aut_X(Y)$ such that $\Aut_X(Z) = \Aut_X(Y)/H$.  So, just as in the
topological case, there is a bijection between \'etale covers
factoring $f$ and subgroups of the automorphism group. In the process
this also shows that every commuting triangle
\[
\xymatrix{
Y\ar[rr]^{h} \ar[dr]_{f} &  & Z \ar[dl]^{g} \\ & X & }
\]
of \'etale covers gives rise to a homomorphism
$\Aut_X(Y)\rightarrow\Aut_X(Z)$ that is, in fact, a quotient map. It
has a concrete description that we will make use of in later
chapters: For every $\varphi\in\Aut_X(Y)$
there is exactly one $\psi\in\Aut_X(Z)$ such that 
\begin{equation}\label{eqn:quotientdefn}
\psi\circ h =
h\circ\varphi;
\end{equation}
the quotient map is defined by $\varphi\mapsto\psi$.

All this invites imposing a partial order on $\operatorname{Fet}(X)$
according to the maps between covers 
and endowing it (along with the accompanying category of automorphism
groups) with the structure of a projective system. Without
base points this is impossible since any morphism of
covering spaces $Y\rightarrow Z$ need not be unique. But if each covering
space $Y$ of $X$ is endowed with a base point $\cl{y}\in
F_{\cl{x}}(Y)$, there is then at most one morphism between
covering spaces that also respects the base points. This provides
both the category of \'etale covers and the category
of the covering groups of \'etale covers with a partial order, giving them the structure of
projective systems, which allows the following definition.

\begin{definition}\label{def:fundamentalgroup}
The \emph{\'etale fundamental group} of a connected scheme $X$ with geometric
point $\cl{x}$ is 
\[
\algfg{X,\cl{x}} \ceq \invlim{Y}\Aut_X(Y)^\mathrm{op},
\]
where the limit is taken over all connected Galois covers of $X$. We
assume the finite groups appearing in the limit are discretely
topolgized, so that $\algfg{X,\cl{x}}$ carries the profinite topology.
\end{definition}

This definition is not actually the one originally given in
SGA~\cite[Exp.~V,\S7]{SGA1}. There, after much heaving of categorical
machinery into place, the fundamental group is defined as the
automorphism group of the \emph{fibre functor} associated to the
chosen geometric point $\cl{x}$,
\begin{align*}
F_{\cl{x}}\colon \Fet{X} & \longrightarrow \operatorname{Ensf} \\
Y & \longmapsto F_{\cl{x}}(Y).
\end{align*}
Our definition is equivalent, and ultimately more agreeable to work
with for being more concrete. The connection between the two
definitions can be seen by considering once again the topological
case, where (under the appropriate conditions) the analogous fibre
functor is represented by a single covering space, the universal
cover. The topological fundamental group is isomorphic to the
automorphism group of that space. Precisely the same thing would be
true in the \'etale case if the projective limit of the Galois covering spaces
of $X$ existed in $\Fet{X}$. It does not, and while one might remedy
this incongruity between the two situations by enlarging the category
from which covers of $X$ can be drawn, in the literature it is more
common let it be, defer taking limits until after passing to
automorphism groups (where the opposite groups are used to ensure that
the resulting group retains a left action on the fibres of the \'etale
covers of $X$) to obtain the \'etale fundamental group, and say
that instead of being representable the fibre functor is
\emph{pro-representable} in the \'etale context.

Viewing the fundamental group as the automorphism group of the fibre
functor does, however, make it easier to see certain key facts; most
importantly, that the fundamental group is functorial. Let $X$ and $Y$
be connected schemes with geometric points $\cl{x}$ and $\cl{y}$, and
$\alpha\colon X\rightarrow Y$ a scheme morphism such that $\alpha
(\cl{x}) = \cl{y}$. Base change by $\alpha$ gives a
  functor $\mathcal{B}_\alpha\colon\Fet{Y}\rightarrow\Fet{X}$, mapping
  any \'etale cover $U$ of $Y$ to $X\times_YU$, as in the Cartesian
  diagram
\[
\xymatrix{
X\times_YU\ar[d] \ar[r]^{\ \ \ \alpha^U} & U \ar[d] \\
X \ar[r]_{\alpha} & Y. }
\]
That $\alpha$ respects the geometric points means that $\alpha^U$
carries any $\cl{u}\in F_{\cl{x}}(X\times_YU)$ to a point in
$F_{\cl{y}}(U)$. Conversely, any point in
$F_{\cl{y}}(U)$ defines, via the universal property, a point in
$F_{\cl{x}}(X\times_YU)$. Therefore, $F_\cl{x} \circ
  \mathcal{B}_\alpha = F_\cl{y}$. So any automorphism of
  $F_\cl{x}$ will give rise to an automorphism of $F_\cl{y}$ via
  $\mathcal{B}_\alpha$, giving a map
\[
\pi_1(\alpha )\colon\algfg{X,\cl{x}}\longrightarrow\algfg{Y,\cl{y}}.
\]
This map will be the true workhorse of
Chapters~\ref{ch:somelocalsystems} and~\ref{ch:mainresult}, so we need
to characterize it explicitly. We can do this by returning to the
view of the fundamental group as an inverse limit.

Let $(\varphi_U)\in\algfg{X,\cl{x}}$, so that $U$ ranges over the
connected Galois covers of $X$. Then the component of
$\algfg{\alpha}((\varphi_U))$ indexed by a connected Galois cover $V$ of
$Y$ is the unique covering transformation $\psi_V$ of $V$ satisfying
\begin{equation}\label{eqn:functorialrelation}
f_V\circ\varphi_{X\times_YV} = \psi_V\circ f_V
\end{equation} 
as in the commuting diagram
\begin{equation}
\xymatrix{ X\times_YV \ar[drr] \ar[rr]^{\varphi_{X\times_YV}} & &
  X\times_YV \ar[d] \ar[rr]^{\alpha^V} & & V \ar[d]
  \ar[rr]^{\psi_V}_{\exists !} & & V
  \ar[dll] \\ & & X \ar[rr]_{\alpha} & & Y & &}
\end{equation}
Such a transformation must exist because $V$ is a Galois cover. 

As a final note, this functoriality property implies that if $\cl{x}$
and $\cl{x}^\prime$ are two geometric points of $X$, then
there is an isomorphism 
\[
\algfg{X,\cl{x}}\stackrel{\sim}{\longrightarrow}
\algfg{X,\cl{x}^\prime}
\]
that is unique up to an inner
automorphism\footnote{See~\cite[Cor. 5.5.2]{szamuely} and the
  subsequent remark for proof.}. The abelianization of
$\algfg{X,\cl{x}}$ is therefore independent of the base point chosen.

\begin{example}\label{ex:fundamentalgroups}\hfill
\begin{enumerate}
\item\label{example:fgofspeck}
  $\Spec{K}$. The previous round of examples described the connected \'etale
  covers of $\Spec{K}$ as the spectra of finite separable
  extensions of $K$. Thus,
\[
\algfg{\Spec{K},\cl{x}} = \invlim{L/K fin.sep.}\Gal{L/K} =
\Gal{K^\mathrm{sep}/K}.
\]
\item\label{example:fgofGm} $\GG_{m,K}$. Let $e\colon\Spec{K}\rightarrow\GG_{m,K}$ be the
  identity of $\GG_{m,K}$ and
  $\cl{e}\colon\Spec{\cl{K}}\rightarrow\Spec{K}\rightarrow\GG_{m,K}$
  and take this to be our geometric point. When $K$ is
  algebraically closed, by Example~\ref{ex:galoiscovers}, we have
\begin{eqnarray*}
\algfg{\GG_{m,K},\cl{e}} & = & \invlim{n\in\NN}\mu_n \\ &\approx&
\invlim{n\in\NN}\ZZ/n\ZZ \\  &=& \widehat{\ZZ}.
\end{eqnarray*}
When $K$ is not algebraically closed, we have instead,
\begin{eqnarray*}
\algfg{\GG_{m,K}} &=& \invlim{n\in\NN}\mu_n\times\Gal{K(\mu_n)/K} \\
&\approx&
\invlim{n\in\NN}\ZZ/n\ZZ\times\invlim{n\in\NN}\Gal{K(\mu_n)/K} \\
& = & \widehat{\ZZ}\times\Gal{(K^\mathrm{sep})^\mathrm{ab}/K}
\end{eqnarray*} 
\item\label{example:fgofA1} $\mathbb{A}^1_K$. From
  Example~\ref{ex:galoiscovers}.3, when $K$ is separably closed
\[
\algfg{\mathbb{A}^1_K,\cl{x}} = 0.
\]
\item $\Spec{\ZZ}$. Likewise,
\[
\algfg{\Spec{\ZZ},\cl{x}} = 0.
\]
\end{enumerate}
\end{example}

Of course, completely enumerating all connected Galois covers of an arbitrary
connected scheme in order to calculate its fundamental group is not a
trivial matter. Fortunately, it is sometimes possible to use knowledge
of topological fundamental groups of complex manifolds to determine
\'etale fundamental groups, courtesy of the next two propositions.

To any scheme $X$ of finite type over $\CC$ it is possible to
associate a complex analytic space $X^\mathrm{an}$ that, when $X$ is
smooth, is a complex manifold (for details, see SGA I~\cite[Expos\'e
XII]{SGA1}). The obvious question then is what relation the topological
fundamental group of $X^\mathrm{an}$ has to the \'etale fundamental
group of $X$. 

\begin{proposition}[SGA I~\cite{SGA1} Expos\'e XII, Cor. 5.2]\label{prop:topfgcomparison}
Let $X$ be a connected scheme of finite type over $\CC$ with
analytification $X^\mathrm{an}$. For every $\CC$-point $\cl{x}\colon
\Spec{\CC}\rightarrow X$ there is an isomorphism
\[
\widehat{\operatorname{\pi_1^\mathrm{top}}(X^\mathrm{an},\cl{x})}\stackrel{\sim}{\longrightarrow}
\algfg{X,\cl{x}},
\]
where the hat denotes the profinite completion of the topological
fundamental group of $X^\mathrm{an}$.
\end{proposition}

Combined with the following proposition, this allows us to use
knowledge of topological fundamental groups of complex manifolds to
determine the \'etale fundamental group of schemes over
$\cl{\QQ}_\ell$, since there exists a (nonunique) isomorphism
$\cl{\QQ}_\ell\cong \CC$.

\begin{proposition}[Szamuely, Prop. 5.6.7]\label{prop:algclosedcomparisonfg}
Let $K\hookrightarrow F$ be an extension of algebraically closed
fields, $X$ a proper integral scheme over $K$, $\cl{x}$ a
geometric point of $X$, and $X_F$ the extension of $X$ to $F$ with
corresponding geometric point $\cl{x}_F$. Then, the functorial
map
\[
\algfg{X_F,\cl{x}_F}\longrightarrow \algfg{X,\cl{x}}
\]
coming from the projection $X_F\rightarrow X$ is an isomorphism.
\end{proposition}

\subsection{Finite Locally Constant Sheaves and the Fundamental Group}

The action of $\algfg{X,\cl{x}}$ on the geometric fibre of any
connected \'etale cover $Y$ comes from an action on the cover
$Y$ itself. Any element of $\gamma\in\algfg{X,\cl{x}}$ can be regarded as a
compatible tuple of covering transformations of \'etale covers of $X$.
The fundamental group acts on $F_{\cl{x}}(Y)$ through the covering
group's action on the fibre:
\[
\gamma\cdot\cl{y} \ceq \varphi_Y\cdot\cl{y} = \varphi_Y(\cl{y}) =
\varphi_Y\circ\cl{y}.
\]

This observation leads to the fundamental result in the theory of the
\'etale fundamental group, and the first step toward the equivalence
between local systems and representations of the fundamental
group:
Let $X$ be a connected scheme over a field $K$
with geometric point $\cl{x}$. The
fibre functor $F_\cl{x}$ defines an equivalence between the category
of \'etale covers of $X$ and the category
$\operatorname{Ensf}(\algfg{X,\cl{x}})$ of finite sets equipped with a continuous left 
action of $\algfg{X,\cl{x}}$. Connected covers correspond to sets with
transitive action, and Galois covers to finite quotients of $\algfg{X,\cl{x}}$.\footnote{A proof can be found in
  Szamuely~\cite[Thm. 5.4.2]{szamuely}.}
Here, `continuous' means with respect to the
profinite topology of the fundamental group. Combined with
Proposition~\ref{prop:coverlocallyconstantcorr} in the case of a
connected scheme $X$, we obtain an equivalence
between $\operatorname{Flc}(X)$ and
$\operatorname{Ensf}(\algfg{X,\cl{x}})$. The composite functor carries
a finite locally constant sheaf to its stalk at
$\cl{x}$. 

This fact can be refined by restricting attention to
the group objects in each of the categories, obtaining the next step
in the road toward the equivalence between local systems and
fundamental group respresentations.

\begin{proposition}\label{prop:locallyconstantsheafcorr}
Let $X$ be a connected scheme over a field $K$
and $\cl{x}\colon \Spec{\cl{K}}\rightarrow X$ a geometric point of
$X$. The functor 
\[
\mathcal{F} \longmapsto \mathcal{F}_{\cl{x}}
\]
defines an equivalence of categories
between finite locally constant sheaves of finite abelian
groups and finite abelian groups equipped with a continuous left action of
$\algfg{X,\cl{x}}$.
\end{proposition}

\subsection{The Homotopy Exact Sequence}\label{subsection:homexactseq}

Finally, we end this section by recalling a crucial exact sequence
associated to the \'etale fundamental group. It will allow us to easily define and keep track of Galois
actions on sheaves and local systems.

Let $X$ be a quasi-compact
  and geometrically integral scheme over a field $K$ with algebraic
  closure $\cl{K}$ and separable closure $K^\mathrm{sep}$. Let
  $\cl{X}$ denote the base extension of $X$ to
  $\Spec{K^\mathrm{sep}}$, and choose a geometric point
    $\cl{x}\colon\Spec{\cl{K}}\rightarrow \cl{X}$ for $\cl{X}$ and
    write $\cl{x}$ for its image in $X$ under the canonical
    projection. With all this, the sequence, called the \emph{homotopy exact
sequence},
\[
1\rightarrow \algfg{\cl{X},\cl{x}}\rightarrow
\algfg{X,\cl{x}}\rightarrow \Gal{K^\mathrm{sep}/K}\rightarrow 1
\]
of profinite groups, where the inner two maps are functorially induced
by the maps $\cl{X}\rightarrow X$ and $X\rightarrow\Spec{K}$, is
exact. On its own this provides a means of determining the fundamental
group of $\mathbb{A}^2_K$ where $K$ is not algebraically
closed. Since, for a separable algebraic closure $\cl{K}$ of $K$,
$\algfg{\mathbb{A}^2_{\cl{K}},\cl{x}} = 0$, it immediately follows
from the homotopy exact sequence that $\algfg{\mathbb{A}^2_K,\cl{x}} =
\Gal{\cl{K}/K}$.

If the scheme $X$ is defined over a field $K$ and has a $K$-rational
point $s$ such that
\[
\xymatrix{
& X \\
\Spec{K} \ar[ur]^s & \\
& \Spec{\cl{K}} \ar[ul] \ar[uu]_{\cl{x}} }
\]
commutes, the functorial map
$\algfg{s}\colon\Gal{K^\mathrm{sep}/K}\rightarrow X$ necessarily gives
a splitting of the homotopy exact sequence, allowing us to write
\[
\algfg{X,\cl{x}} \iso
\algfg{\cl{X},\cl{x}}\rtimes\Gal{K^\mathrm{sep}/K}.
\]
Using the identity $e$, this fact gives a different way of seeing that the fundamental group
of $\GG_{m,K}$ is as described in Example~\ref{ex:fundamentalgroups}.\ref{example:fgofGm}.

\section{$\ell$- and $\pi$-adic Sheaves}


This section is devoted to recalling the definition and basic facts
about the category of $\pi$-adic sheaves, which are fundamental to the
definition of $\cl{\QQ}_\ell$-local systems. Most important is the
equivalence between the category of $\cl{\QQ}_\ell$-local systems and
representations of the \'etale fundamental group. Whenever possible,
we will make use of this equivalence to avoid the cumbersome work of
dealing with the local systems themselves. 

The definitive source for $\ell$-adic sheaves remains Expos\'es V and
VI of SGA 5~\cite{SGA5}. Freitag and Kiehl's book~\cite{freitagkiehl} provides most of
the information here in greater depth for the case of $\ell$-adic
sheaves, which $\pi$-adic sheaves generalize in a straightforward
way. All of the results there remain valid for $\pi$-adic sheaves. The
book by Kiehl and Weissauer~\cite{kiehlweissauer} discusses the extension of some deeper
results to $\cl{\QQ}_\ell$-sheaves. Another good
source is section 1.4 of Conrad's notes, \emph{\'Etale Cohomology}~\cite{conradetale}.

A projective system of sheaves on the \'etale site of $X$ 
\[
\cdots\rightarrow \mathcal{F}_{n+1}\rightarrow\mathcal{F}_n\rightarrow
\mathcal{F}_{n-1}\rightarrow\cdots
\]
will be written $\mathcal{F} = (\mathcal{F}_n)_{n\in\ZZ}$ or just
$(\mathcal{F}_n)$,  and a morphism between two
systems $f\colon (\mathcal{F}_n)\rightarrow (\mathcal{G}_n)$
\[
\xymatrix{
\cdots \ar[r] & \mathcal{F}_{n+1} \ar[r] \ar[d]^{f_{n+1}} &
\mathcal{F}_n \ar[r] \ar[d]^{f_n} & \mathcal{F}_{n-1} \ar[r]
\ar[d]^{f_{n-1}} & \cdots \\
\cdots \ar[r] & \mathcal{G}_{n+1} \ar[r] & \mathcal{G}_n \ar[r] &
\mathcal{G}_{n-1} \ar[r] & \cdots }
\]
similarly as $f = (f_n)$ when the maps between individual components
need emphasis. Finally, let $E$ be a finite extension of
the $\ell$-adic numbers, $\QQ_\ell$, with valuation ring $\mathcal{O}_E$
and uniformizer $\pi$.

Throughout this section let $X$ be a Noetherian scheme on which the
prime $\ell$ is invertible (meaning that $\ell$ maps to an invertible
element in every ring of sections of the structure sheaf of $X$). We
make the Noetherian requirement because the definition of $\pi$-adic sheaves
uses constructibility, which is a simpler concept on Noetherian
schemes. An \'etale sheaf $\mathcal{F}$ on a Noetherian scheme $X$ is
\emph{constructible} if $X$ can be written as a finite union of
locally closed subschemes such that the restriction of $\mathcal{F}$
to each locally closed stratum is finite locally constant.

\subsection{$\pi$-adic Sheaves}

A \emph{$\pi$-adic sheaf} on $X$ is a projective system of \'etale
sheaves of $\mathcal{O}_E$-modules $(\mathcal{F}_n)_{n\in\ZZ}$ satisfying
\begin{itemize}
\item[(i)] $\mathcal{F}_n$ is constructible for all $n$,
\item[(ii)] $\mathcal{F}_n = 0$ for $n < 0$,
\item[(iii)] $\pi^{n+1}\mathcal{F}_n = 0$ for $n\geq 0$,
\item[(iv)]
  $\mathcal{F}_{n+1}\otimes_{(\mathcal{O}_E/\pi^{n+2}\mathcal{O}_E)}(\mathcal{O}_E/\pi^{n+1}\mathcal{O}_E)
  \iso \mathcal{F}_n$ for all $n\geq 0$.
\end{itemize}
When $\mathcal{O}_E = \ZZ_\ell$ and $\pi = \ell$, $(\mathcal{F}_n)$ is
called an $\ell$-adic sheaf. A $\pi$-adic sheaf is \emph{locally
  constant} if $\mathcal{F}_n$ is locally constant for
every $n$. If $\cl{x}$ is a geometric point of $X$, the \emph{stalk}
of $\mathcal{F} = (\mathcal{F}_n)$ at $\cl{x}$ is 
\[
\mathcal{F}_\cl{x}\ceq \invlim{n\in\NN}(\mathcal{F}_n)_\cl{x}.
\]

Despite the name, it is not true that there exists an etale covering
on which a locally constant $\pi$-adic sheaf becomes constant (\ie , all of the constituent sheaves
$\mathcal{F}_n$ become constant simultaneously) after restriction. An easy instance where this fails is the Tate twist
$\ZZ_\ell(1)$ described in Example~\ref{ex:piadicexamples}. Therefore, the term ``locally
constant $\pi$-adic sheaf'' is rarely used in the
literature. Unfortunately, the terminology that is in use is
confusingly varied. In SGA, locally constant $\ell$-adic sheaves are called
\emph{constant tordu}. Now they are more commonly called \emph{lisse}
or \emph{smooth}. We will call them smooth.

There is also the completely analogous concept of a \emph{$\pi$-adic system
of $\mathcal{O}_E$-modules} $M = (M_n)$ for which the defining properties are
\begin{itemize}
\item[(i)] $M_n$ is a module of finite length for all $n$,
\item[(ii)] $M_n = 0$ for $n < 0$,
\item[(iii)] $\pi^{n+1} M_n = 0$ for $n\geq 0$,
\item[(iv)]
  $M_{n+1}\otimes_{\mathcal{O}/\pi^{n+2}\mathcal{O}}(\mathcal{O}_E/\pi^{n+1}\mathcal{O}_E)
  \iso M_n$ for all $n\geq 0$.
\end{itemize}
Transporting entire $\pi$-adic sheaves using the equivalence in
Proposition~\ref{prop:locallyconstantsheafcorr} results in a $\pi$-adic
system of $\mathcal{O}_E$-modules.

\begin{example}\label{ex:piadicexamples}
\begin{itemize}
\item[(i)] The constant sheaf $\ell$-adic sheaf $(\ZZ_\ell )_X$ is
\[
\cdots\rightarrow (\ZZ/\ell^2\ZZ )_X\rightarrow (\ZZ/\ell\ZZ )_X \rightarrow 0_X.
\]
More generally, the constant sheaf $(\mathcal{O}_E)_X$ (not to be
confused with the structure sheaf $\mathcal{O}_X$) is similarly a $\pi$-adic sheaf.
\item[(ii)]\label{ex:tatetwist} Let $(\mu_{\ell^n})_X$ be the sheaf of $\ell^n$-th roots of
  unity on $X$ (that is, the kernel of the $\ell^n$ power map on
  $\mathcal{O}^*_X$, the sheaf of units of the structure sheaf of
  $X$). Then, the projective system of these sheaves,
  with $(\mu_{\ell^{n+1}})_X\rightarrow (\mu_{\ell^n})_X$ the $\ell$th
  power map, is an $\ell$-adic sheaf, the \emph{Tate twist} of the constant
  sheaf, $(\ZZ_\ell(1))_X$. More generally, the $m$th Tate twist,
  $(\ZZ_\ell(m))_X \ceq ((\mu_{\ell^n})_X^{\otimes m})$, is an
  $\ell$-adic sheaf.

The Tate twist provides an easy example of the ``failure of local
constancy'' noted above. Suppose $X =
\Spec{K}$, $K$ a field. Then, the sheaves $(\mu_{\ell^n})_X$ would
only become constant simultaneously over an \'etale neighbourhood
defined by a finite extension of $K$ containing all $\ell$-power roots
of unity. Of course, no such field exists. But each component sheaf
does become constant individually over an appropriately large cyclotomic
extension of $K$, so $(\ZZ_\ell(1))_X$ is a locally constant or, as we
prefer, smooth, $\pi$-adic sheaf.
\end{itemize}
\end{example}

\subsection{$E$-Local Systems}

Smooth $\pi$-adic sheaves give rise to local systems after
being embedded in the category of sheaves of $E$-vector spaces. To define those,
we need to introduce the Artin-Rees, or A-R, category first. This
category arises out of an effort to construct a category that carries
over the abelian category structures and operations (kernels,
cokernels, $\Hom$, and so on) of sheaves of
modules to the projective system category by applying them term by term. As
defined, this isn't possible in the category of $\pi$-adic sheaves itself,
since, for example, the projective system of kernels $(\kernel{f_n})$
of a morphism
\[
\xymatrix{
\cdots \ar[r] & \mathcal{F}_{n+1} \ar[r] \ar[d]^{f_{n+1}} &
\mathcal{F}_n \ar[r] \ar[d]^{f_n} & \mathcal{F}_{n-1} \ar[r]
\ar[d]^{f_{n-1}} & \cdots \\
\cdots \ar[r] & \mathcal{G}_{n+1} \ar[r] & \mathcal{G}_n \ar[r] &
\mathcal{G}_{n-1} \ar[r] & \cdots }
\]
does not necessarily satisfy condition (iii) of the definition. The expanded A-R category addresses
this\footnote{Conrad's notes~\cite{conradetale} are more expansive on this topic if the reader
is interested.}.

Write $\mathcal{F}[r]$ for the projective system $\mathcal{F} = (\mathcal{F}_n)_{n\in\NN}$ shifted $r$ degrees; \ie,
$\mathcal{F}[r] = (\mathcal{F}_{n+r})_{n\in\ZZ}$. The objects of the A-R category of 
sheaves are all projective systems $(\mathcal{F}_n)$ whose
component sheaves are $\pi$-torsion (\ie, every section is
annihilated by some power of $\pi$). The morphisms in
the category are given by
\[
\Hom_{A-R}(\mathcal{F},\mathcal{G})\ceq
\invlim{r}\Hom(\mathcal{F}[r],\mathcal{G}).
\]
Any projective system in the A-R category that is isomorphic to an $\pi$-adic
sheaf is called \emph{A-R $\pi$-adic}. With this, we can begin to
develop the notion of a local system.

The category of \emph{sheaves of $E$-vector spaces} on $X$ has the same
objects as the Artin-Rees category of $\pi$-adic sheaves, written
$\mathcal{F}\otimes E$ to distinguish them. For any two sheaves of
$E$-vector spaces $\mathcal{F}\otimes E$ and $\mathcal{G}\otimes E$,
\[
\Hom(\mathcal{F}\otimes E,\mathcal{G}\otimes E) \ceq
\Hom_{A-R}(\mathcal{F},\mathcal{G})\otimes_{\mathcal{O}_E} E.
\]

Local systems are special objects within this category. Specifically,
an \emph{$E$-local system} $\mathcal{L}$ on $X$ is a sheaf
of $E$-vector spaces isomorphic to $\mathcal{F}\otimes E$, where
$\mathcal{F}$ is a smooth $\pi$-adic sheaf. Smoothness means that if
$X$ is connected and
$\cl{x}$ and $\cl{x}^\prime$ are any two geometric points of $X$,
the stalks of $\mathcal{L}$ at $\cl{x}$ and $\cl{x}^\prime$ are
isomorphic. In particular, they are $E$-vector spaces of the same
dimension, and this common dimension is the \emph{rank} of
$\mathcal{L}$.


In the previous section we noted that there was an equivalence between
finite locally constant sheaves of abelian groups and finite abelian
groups with a continuous action of the \'etale fundamental group.
That equivalence can be extended to the objects now at hand. First,
by applying Proposition~\ref{prop:locallyconstantsheafcorr} to each
component sheaf, we see that a smooth $\pi$-adic sheaf is equivalent
to a $\pi$-adic system of $\mathcal{O}_E$-modules 
\[
\cdots\rightarrow M_{n+1}\rightarrow M_{n}\rightarrow
M_{n-1}\rightarrow\cdots ,
\]
where each $M_i$ is an $\mathcal{O}_E/\pi^i\mathcal{O}_E$-module
equipped with a continuous action of $\algfg{X,\cl{x}}$. 
In SGA 5 it is shown that the inverse limit
functor defines an equivalence between the category of $\pi$-adic modules and the
category of finitely generated $\mathcal{O}_E$-modules with continuous
fundamental group action. By composition, we obtain:

\begin{proposition}\label{prop:piadiccorrespondence}[SGA 5~\cite{SGA5}, Exp. VI,
    Lemme 1.2.4.2]
The category of smooth $\pi$-adic sheaves on a scheme $X$ is
equivalent to the category of finitely generated
$\mathcal{O}_E$-modules equipped with a continuous action of the \'etale fundamental
group of $X$.
\end{proposition}

Finally, the obvious functor from the category of $\pi$-adic sheaves
to the category of sheaves of $E$-vector spaces is mirrored in the
module category by the ``tensor with $E$'' functor
\begin{eqnarray*}
M &\longmapsto & M\otimes_{\mathcal{O}_E}E
\end{eqnarray*}
into the category of $E$ vector spaces, providing an equivalence
between the category of $E$-local systems on a connected scheme $X$ with
the category of $E$-valued representations of
$\algfg{X,\cl{x}}$\footnote{Conrad sketches a proof
  in~\cite[Thm. 1.4.5.4]{conradetale}}.  Under this equivalence, the
rank of a local system is equal to the dimension of the corresponding representation.



\subsection{$\cl{\QQ}_\ell$-Local Systems}

Finally, we will define the category of $\cl{\QQ}_\ell$-local
systems, which is built up as a direct limit of the categories of
$E$-local systems. To see how\footnote{Conrad~\cite{conradetale} takes
  a different approach to this construction; see Definition 1.4.5.6.}, note that for any finite extension of
local fields, $E\subseteq L$, with rings of integers $\mathcal{O}_E$
and $\mathcal{O}_L$ respectively, and uniformizers $\pi_E$ and
$\pi_L$ such that $\pi_L^e = \pi_E$, where $e$ is the ramification
index of the extension, there is a functor from the category of
$\pi_E$-adic sheaves of $X$ to that of $\pi_L$-adic sheaves on $X$
defined by
\[
\mathcal{F} = (\mathcal{F}_n) \longmapsto \cl{\mathcal{F}} =
(\cl{\mathcal{F}}_n),
\]
where 
\[
\cl{\mathcal{F}}_{ne}\ceq
\mathcal{F}_n\otimes_{(\mathcal{O}_E/\pi_E^{n+1}\mathcal{O}_E)}
\mathcal{O}_L/\pi_L^{(n+1)e}
\]
for all $n$ and for $1\leq i < e$,
\[
\cl{\mathcal{F}}_{ne-i}\ceq
\mathcal{F}_n\otimes_{(\mathcal{O}_E/\pi_E^{n+1}\mathcal{O}_E)}
\mathcal{O}_L/\pi_L^{(n+1)e-i}\mathcal{O}_L .
\]
This functor can be extended to one from the category of sheaves of
$E$-vector spaces on $X$ to the category of $L$-vector spaces by 
\[
\mathcal{F}\otimes E \longmapsto \cl{\mathcal{F}}\otimes L.
\]
From this, we get a directed system of categories, the limit of which
is the category of sheaves of $\cl{\QQ}_\ell$-vector spaces. Every
object $\mathcal{F}\otimes\cl{\QQ}_\ell$ in the category can be represented by a sheaf of $E$-vector
spaces,
\[
\mathcal{F}\otimes E,
\]
and if there exists such a representative that is an $E$-local system,
so that $\mathcal{F}$ is a smooth $\pi_E$-adic sheaf, 
then we say that $\mathcal{F}\otimes\cl{\QQ}_\ell$ is a
\emph{$\cl{\QQ}_\ell$-local system}. 

As one would expect, the equivalence for $E$-local systems extends to
$\cl{\QQ}_\ell$-local systems.

\begin{proposition}\label{prop:ladiclscorr}
The category of sheaves of $\cl{\QQ}_\ell$-local systems (or,
alternately, $\ell$-adic local systems) on a connected
scheme $X$ is equivalent
to the category of finite dimensional $\cl{\QQ}_\ell$ representations
of $\algfg{X,\cl{x}}$. The rank of a local system under this
equivalence os equal to the dimension of the corresponding
representation.
\end{proposition}
Conrad's notes give a proof~\cite[Thm.~1.4.5.7]{conradetale}.

\subsection{Proper and Smooth Base Change}

Before closing this chapter, we state two important theorems for
reference. These are the base change theorems for
proper and smooth morphisms, originally proved in SGA for torsion
sheaves, which are sheaves of groups whose every group of sections is
a torsion group. Both theorems extend from that case to the case of $\pi$-adic sheaves
and $\cl{\QQ}_\ell$-local systems.

\begin{theorem}[Proper Base Change, SGA 4(III)~\cite{SGA4iii}, Expos\'e XII,
  5.1(i)]\label{thm:pbc}
Let $X$ and $Y$ be schemes and $f\colon X\rightarrow Y$ a proper
morphism. Let
\[
\xymatrix{
X\times_YW\ar[r]^{\ \ \ f^\prime} \ar[d]_{g^\prime} & W \ar[d]^{g} \\
X \ar[r]_{f} & Y }
\]
be a Cartesian diagram. Then, for every torsion sheaf (or $\pi$-adic
sheaf, or $\cl{\QQ}_\ell$-sheaf) $\mathcal{F}$ on $X$, the base change
isomorphism
\[
g^*(R^if_*\mathcal{F})\stackrel{\sim}{\longrightarrow}
R^if^\prime_*({g^\prime}^*\mathcal{F})
\]
is an isomorphism.
\end{theorem}
\begin{theorem}[Smooth Base Change, SGA 4(III)~\cite{SGA4iii}, Expos\'e
  XVI]\label{thm:sbc}
Suppose
\[
\xymatrix{
X\times_YW\ar[r]^{\ \ \ f^\prime} \ar[d]_{g^\prime} & W \ar[d]^{g} \\
X \ar[r]_{f} & Y }
\]
is a Cartesian diagram of Noetherian schemes. If $g$ is a smooth
morphism and $\mathcal{F}$ is a torsion sheaf ($\pi$-adic or $\cl{\QQ}_\ell$-sheaf
with torsion components) whose sections have order relatively prime
to the residue characteristic of $Y$, then the base change
homomorphism
\[
g^*(R^if_*\mathcal{F})\longrightarrow
R^if^\prime_*({g^\prime}^*\mathcal{F})
\]
is an isomorphism.
\end{theorem}

   \chapter{Galois Actions}\label{ch:galoisactions}

In SGA 7(II)~\cite{SGA7ii} Expos\'e XII, Deligne introduced the notion of a ``Galois
sheaf''\footnote{Deligne does not offically define any objects with
  this name however since, as noted below, the notion is slightly
  broader. In light of the actual use he puts them to, this
  seems a reasonable name.} on a scheme $X$ over a field $K$, and
proved that the category of Galois sheaves on the extension
$\cl{X}\ceq X\times_{\Spec{K}}\Spec{\cl{K}}$ to a separable closure
$\cl{K}$ of $K$ is equivalent to the category of sheaves on $X$. 
There are indications in the literature that the same kind of
correspondence exists more generally. In particular, the remark
following Proposition 5.1.2 in~\cite{BBD} states that it is true of
complexes belonging to the bounded derived category of $\ell$-adic
sheaves. However, no indication of a proof is given there, nor, it
seems, does one exist anywhere else. Since the results of
Chapter~\ref{ch:mainresult} are formulated around the existence of
this kind of correspondence, we give a proof here.

In this chapter we show that the sort of Galois action Deligne
describes can easily be extended to $\pi$-adic sheaves and
$\cl{\QQ}_\ell$-local systems. Less trivially, in the case of
$\cl{\QQ}_\ell$-local systems, it turns out that a
splitting of the homotopy exact sequence (as in Section~\ref{subsection:homexactseq})
allows us to naturally view the data of a Galois action in Deligne's
sense as a part of the representation of the \'etale
fundamental group equivalent to the local system. This gives a means
of extending the equivalence proved in SGA for torsion sheaves (see
Proposition~\ref{prop:delignegaloissheaf}) to a
class of $\cl{\QQ}_\ell$-local systems that includes those we define
in Chapter~\ref{ch:somelocalsystems}. And although the approach taken here may not
establish this fact in its fullest generality, it does allow us
to pursue our ongoing policy of working with fundamental group
representations rather than the $\cl{\QQ}_\ell$-local systems to which
they correspond.

We begin with a recollection of Deligne's original Galois
sheaf concept.

\section{Deligne's Galois Sheaves}

As mentioned, all of the material here first appeared in SGA 7(II)~\cite{SGA7ii}
Expos\'e XII, although the presentation here is
more detailed in certain respects so that this notion of a Galois
action on a sheaf can be clearly compared with the Galois action
implicitly contained in the representation of a local system on $X$ as a
representation of the fundamental group of $X$ when $X$ has a
$K$-rational point. A thorough review of the details
is also advisable because, despite the name, Deligne's notion of a
Galois action on a sheaf extends to any profinite group equipped with
a continuous homomorphism to the Galois group in question, and this
provides a crucial structural aspect of the nearby cycles functor,
which is vital to the results of Chapter~\ref{ch:mainresult}.

Let $X$ be a scheme over a field $K$ with separable closure $\cl{K}$,
and let $\cl{X}\ceq X\times_{\Spec{K}}\Spec{\cl{K}}$. To begin, we
allow $G$ to be a profinite group equipped with a continuous homomorphism
\[
\epsilon\colon G\longrightarrow \Gal{\cl{K}/K}.
\]
The Galois group naturally acts on $\cl{X}$ via its action on
$\cl{K}$, and we will denote the induced action of
$\gamma\in\Gal{\cl{K}/K}$ by $\gamma_X$. This is visualized in the
diagram
\[
\xymatrix{
\cl{X}\ar[d] \ar[r]^{\gamma_X}_\sim & \cl{X} \ar[d] \ar[r] & X \ar[d] \\
\Spec{\cl{K}}\ar[r]_{\Spec{\gamma}}^\sim & \Spec{\cl{K}}\ar[r] & \Spec{K} }
\]
where all squares in sight are Cartesian. The group $G$ therefore acts
on $\cl{X}$ via the homomorphism $\epsilon$.

\begin{definition}\label{def:compatibleaction}
Let $\mathcal{F}$ be an \'etale sheaf of sets on $\cl{X}$. An \emph{action
of $G$ on $\mathcal{F}$ compatible with the action of $G$ on $\cl{X}$}
is a system of isomorphisms, one for each $g\in G$,
\[
\mu (g)\colon\epsilon (g)_*\mathcal{F}\stackrel{\sim}{\longrightarrow}\mathcal{F}
\]
that satisfy, for any $h\in G$, $\mu (g\circ h) = \mu (g)\circ
\epsilon(g)_*(\mu (h))$.
\end{definition}
This idea takes the shape of a recognizable group action when we
consider how the $\mu (g)$ act on sections of $\mathcal{F}$ over an
\'etale scheme $\cl{U}$ arising via base change from an \'etale
scheme $U$ over $X$. In that case, the pullback of the diagram
\[
\xymatrix{
 & \cl{U}\ar[d] \\
\cl{X}\ar[r]_{\gamma_X} & \cl{X} }
\]
is again $\cl{U}$ and so
\[
(\mu (g))(\cl{U})\colon \epsilon (g)_*\mathcal{F}(\cl{U}) =
\mathcal{F}(\cl{U}) \stackrel{\sim}{\longrightarrow}
\mathcal{F}(\cl{U})
\]
is a permutation of the sections of $\mathcal{F}$ over $\cl{U}$.
\begin{definition}\label{def:ctsaction} 
A profinite group $G$ \emph{acts continuously} on $\mathcal{F}$ if,
for any $U$ quasicompact and \'etale over $X$, $G$ acts continuously
on the discrete set $\mathcal{F}(\cl{U})$.
\end{definition}

The fundamental example of such an action arises like so: If, instead
of $\cl{X}$, $\mathcal{F}$ is a sheaf of sets on $X$, and
$\cl{\mathcal{F}}$ is the inverse image of $\mathcal{F}$ under the
projection $\cl{X}\rightarrow X$, then $\cl{\mathcal{F}}$ naturally
carries an action of $\Gal{\cl{K}/K}$ that is compatible with the
action on $\cl{X}$. For every $\gamma\in\Gal{\cl{K}/K}$, we have the
Cartesian square
\[
\xymatrix{
\cl{U}\times_\cl{X}\cl{X} \ar[r]^{\ \ \ \sim} \ar[d] & \cl{U} \ar[d] \\
\cl{X} \ar[r]_{\gamma_X} & \cl{X}. }
\]
The isomorphism $\mu (\gamma )\colon (\gamma_X)_*\mathcal{F}(\cl{U})
\stackrel{\sim}{\rightarrow} \mathcal{F}(\cl{U})$ is simply the
restriction map $\cl{\mathcal{F}}(\cl{U}\times_{\cl{X}}\cl{X}\rightarrow\cl{U})$.
This action is continuous in the sense above, which can be
proved by considering the situation at the finite level: let $L$ be a
finite extension of $K$ and write $X_L$ for the base extension of $X$
to $L$ and $\mathcal{F}_L$ for the inverse image of $\mathcal{F}$
under $X_L\rightarrow X$. Then $\cl{\mathcal{F}} =
\invlim{L}\mathcal{F}_L$ and $\Gal{\cl{K}/K}$ acts on $\mathcal{F}_L$,
through the finite quotient $\Gal{L/K}$, so that the action is
necessarily continuous. Since the morphisms giving the action on
$\cl{\mathcal{F}}$ are inverse limits of the morphisms at the finite
level, it is also continuous.

There is therefore a functor from sheaves on $X$ to Galois sheaves on
$\cl{X}$, which is in fact one half of an equivalence, as Deligne proved.

\begin{proposition}\label{prop:delignegaloissheaf}[SGA 7(II)~\cite{SGA7ii} Expos\'e
  XII, Rappel 1.1.3(i)]
The category of sheaves of sets on the \'etale site of $X$ is
equivalent to the category of sheaves on $\cl{X}$ equipped with a
continuous action of $\Gal{\cl{K}/K}$ compatible with the action on
$\cl{X}$. 
\end{proposition}

The quasi-inverse of the functor described above is given, predictably, by invariants. Let $\pi$ be the
projection $\cl{X}\rightarrow X$. The pushforward
$\pi_*\cl{\mathcal{F}}$ inherits an action of $\Gal{\cl{K}/K}$ from
$\cl{\mathcal{F}}$ that is a straightforward action on the set of
sections over every \'etale scheme over $X$, as described
above. Therefore, it makes sense to take invariants under this action,
and we define
\[
\cl{\mathcal{F}}^{\Gal{\cl{K}/K}} \ceq
\left(\pi_*\cl{\mathcal{F}}\right)^{\Gal{\cl{K}/K}}.
\]

\section{Galois Actions on $\pi$-adic Sheaves and Local Systems}\label{section:systemactions}

Definition~\ref{def:compatibleaction} can be naturally extended to the
category of $\pi$-adic sheaves by taking the morphisms $\mu (g)$ to be
morphisms in that category, with the pushforward by $\epsilon (g)$
taken componentwise:
\[
\xymatrix{
\cdots \ar[r] & \epsilon (g)_*\mathcal{F}_{n+1} \ar[r] \ar[d]^{{\mu (g)}_{n+1}} &
\epsilon (g)_*\mathcal{F}_n \ar[r] \ar[d]^{{\mu (g)}_n} & \epsilon (g)_*\mathcal{F}_{n-1} \ar[r]
\ar[d]^{{\mu (g)}_{n-1}} & \cdots \\
\cdots \ar[r] & \mathcal{F}_{n+1} \ar[r] & \mathcal{F}_n \ar[r] &
\mathcal{F}_{n-1} \ar[r] & \cdots }.
\]
From there, sheaves of $E$-vector spaces and $\cl{\QQ}_\ell$-sheaves
can also be endowed with profinite group actions. The notion of a
continuous action also extends to these categories by demanding that
${\mu (g)}_n$ satisfy Definition~\ref{def:ctsaction} for all $n$. 


Our goal now is to prove a version of
Proposition~\ref{prop:delignegaloissheaf} for $\cl{\QQ}_\ell$-local
systems on a scheme $X$ over a field $K$ with a $K$-rational point. By
Proposition~\ref{prop:ladiclscorr}, we can represent a rank 1 $\cl{\QQ}_\ell$-local system with an
$\ell$-adic character of the \'etale fundamental group of $X$. The
claim is that the homotopy exact sequence recalled in Section~\ref{subsection:homexactseq} 
in fact presents us with Deligne's equivalence for these local
systems on $X$.

Let $e\colon\Spec{K}\rightarrow X$ be a $K$-rational point of $X$, and
$\cl{e}$ the geometric point corresponding to a choice of embedding
$K\hookrightarrow \cl{K}$. Then, as described in
Section~\ref{subsection:homexactseq},
the functorial map 
\[
\algfg{e}\colon \algfg{\Spec{K},\cl{e}} =
\Gal{\cl{K}/K}\rightarrow \algfg{X,\cl{e}}
\]
necessarily splits the homotopy exact sequence
\[
1\rightarrow\algfg{\cl{X},\cl{e}}\rightarrow\algfg{X,\cl{e}}\rightarrow\Gal{\cl{K}/K}\rightarrow
1,
\]
which allows us to write $\algfg{X,\cl{e}} =
\algfg{\cl{X},\cl{e}}\rtimes\Gal{\cl{K}/K}$. Because our attention is
restricted to rank 1 $\cl{\QQ}_\ell$-local systems and thus to $\ell$-adic
characters of $\algfg{X,\cl{e}}$, which must factor through the
respective abelianizations of the groups involved, we can consider the
exact sequence
\[
1\rightarrow\algfg{\cl{X},\cl{e}}_{\Gal{\cl{K}/K}}\rightarrow
\algfg{\cl{X},\cl{e}}_{\Gal{\cl{K}/K}}\times \Gal{K^\mathrm{ab}/K}
\rightarrow \Gal{K^\mathrm{ab}/K}\rightarrow 1.
\]
associated to the abelianization of $\algfg{\cl{X},\cl{e}}\rtimes\Gal{\cl{K}/K}$.
This information alone tells us something very similar to
Proposition~\ref{prop:delignegaloissheaf}---that every rank 1 $\ell$-adic
local system on $X$ is equivalent to a rank 1 local system on
$\cl{X}$ together with a Galois action on
$\cl{\QQ}_\ell$, which is isomorphic to the stalks of the local system
on $\cl{X}$. It is also isomorphic to the stalks of the local system
on $X$ as well, and what makes the claim less than immediate is the fact
that the action indicated is in fact on the stalks of the local system
there, not on $\cl{X}$. To prove the claim, we need to show that this
Galois action can be translated to one as in
Definition~\ref{def:compatibleaction}.


Let $\rho\colon\algfg{X,\cl{x}}\rightarrow\cl{\QQ}_\ell^\times$ be a
character and $\mathcal{F}$ a $\cl{\QQ}_\ell$-local system equivalent
to $\rho$ under the correspondence of
Proposition~\ref{prop:ladiclscorr}. We may assume that
$\mathcal{F}$ is itself represented by an $E$-local system
$\mathcal{F}\otimes E$, where $\mathcal{F}$ is a smooth $\pi_E$-adic
sheaf and $E$ is a finite extension of $K$. Each component torsion
sheaf $\mathcal{F}_n$ of $\mathcal{F}$ may be taken to be the sheaf of
sections of some Galois cover of $X$ by the equivalence of finite
locally constant sheaves and \'etale covers.

It is sufficient to prove the claim in the case of smooth $\pi$-adic sheaves to
establish it for $E$-local systems since the latter category is just a
quotient of the former and because the Galois action on an $E$-local
system is defined through the $\pi$-adic sheaf giving rise to
it. Likewise, the claim follows for $\cl{\QQ}_\ell$-local systems from the
case of $E$-local systems. 

We can make one further reduction and prove the claim for each
constituent sheaf $\mathcal{F}_n$, since the fundamental group action
is defined for each individually, arising from the action of the
covering group on the Galois cover of $X$ corresponding to
$\mathcal{F}_n$ under Proposition~\ref{prop:coverlocallyconstantcorr}.
We are left with the following statement.
\begin{lemma}\label{lemma:galoisactionexchange}
Let $F_n$ be a finite locally constant sheaf on a scheme $X$, with a
$K$-rational point $e$, equivalent to a
character 
\[
\rho_n\colon \algfg{X,\cl{e}} = \algfg{\cl{X},\cl{e}}_{\Gal{\cl{E}/E}}
\times \Gal{E^{\textrm{ab}}/E} \rightarrow \cl{\QQ}_\ell^\times .
\]
Then, the natural Galois action on $\cl{\mathcal{F}}_n$ (in the sense of
Definiton~\ref{def:compatibleaction}) arising from base change to
$\cl{X}$ is equal to that defined by the action of $\algfg{X,\cl{x}}$
on $\mathcal{F}_n$ and the inclusion
\[
\Gal{\cl{E}/E} \stackrel{\pi_1(e)}{\rightarrow} \algfg{\cl{X},\cl{e}}
\rtimes \Gal{\cl{E}/E}.
\]
\end{lemma}
\begin{proof}
We begin by calling into service the
equivalence of Proposition~\ref{prop:coverlocallyconstantcorr}, which
furnishes us with an \'etale cover $Y_n$ whose sheaf of sections is
the finite locally constant sheaf $\mathcal{F}_n$.

First, recall the action of the Galois group induced by the action
of the fundamental group. The splitting of the fundamental group
\[
\algfg{X,\cl{e}} = \algfg{\cl{X},\cl{e}} \rtimes \Gal{\cl{E}/E}
\]
comes from the functorial map
\[
\algfg{e}\colon \Gal{\cl{E}/E}\longrightarrow \algfg{X,\cl{e}}
\]
where $e$ is the $K$-rational point of $X$. Let $E_n$ be the finite Galois
extension of $E$ defined by the Cartesian square
\[
\xymatrix{
\Spec{E_n} \ar[rr]^{e_n} \ar[d] & & Y_n  \ar[d] \\
\Spec{E} \ar[rr]_{e} & & X }
\]
Given $\gamma\in\Gal{\cl{E}/E}$,  $\pi_1(e)$ maps $\gamma$ to the
unique covering transformation $\varphi_n$ of $Y_n$ such
that $e_n\circ\Spec{\gamma} = \varphi_n\circ e_n$. Let $U$ be
\'etale over $X$. Then, $\varphi_n$ induces a transformation of the cover
$U\times_XY_n\rightarrow U$ that we will denote $\varphi_n^U$. If
$\sigma\colon U\rightarrow U\times_XY_n$ is a section of the cover, so is $\varphi_N^U\circ\sigma$.
Therefore $\sigma\mapsto \varphi_n^U\circ\sigma$ defines a permutation
of the sections of $\mathcal{F}_n$ over $U$. The action of the Galois
group coming from the fundamental group action is defined through this action:
\begin{equation}\label{eqn:fgGaloisaction}
\gamma\cdot\sigma \ceq \varphi_n^U\cdot\sigma = \varphi_n^U\circ\sigma.
\end{equation}

Next we consider the natural Galois sheaf structure obtained by base change.
As explained in the discussion preceding
Proposition~\ref{prop:delignegaloissheaf}, the isomorphisms giving this action are
inverse limits of isomorphisms obtained by base changing to finite
Galois extensions of $E$. For the sake of proper comparison, we retain
the finite Galois extension $E_n$ from above. Again, let
$\gamma\in\Gal{\cl{E}/E}$ and abusively write $\gamma$ for its
restriction to $E_n$ as well. Let $U$ be \'etale over $X$ and
$U_{E_n}$ its base change to $E_n$. Consider
the following diagram, in which all squares are Cartesian.
\[
\xymatrix{
U_{E_n}\times_XY_n \ar[r]_{\gamma_{U\cap Y_n}}^{\sim} \ar[d] & U_{E_n}\times_XY_n \ar[d] \ar[r] &
Y_n \ar[d] \\
U_{E_n} \ar[r]_{\gamma_U}^{\sim} \ar[d] & U_{E_n} \ar[r] \ar[d] & X \ar[d] \\
\Spec{E_n} \ar[r]_{\Spec{\gamma}}^{\sim} & \Spec{E_n} \ar[r] & \Spec{E} }
\]
Set $X_{E_n} =
X\times_\Spec{E}\Spec{E_n}$ and let $(\mathcal{F}_n)_{X_{E_{n}}}$ be
the inverse image of $\mathcal{F}_n$ on $X_{E_n}$. The restriction of
the Galois sheaf isomorphism to this level
\[
\mu(\gamma)(U)\colon \gamma_*\mathcal{F}_n(U)
\stackrel{\sim}{\longrightarrow} \mathcal{F}_n(U)
\]
is simply the restriction map
$\left(\mathcal{F}_n\right)_{X_{E_n}}(\gamma_U)$ associated to $\gamma_U$. 
To complete the comparison with the other
action, we need to explicitly describe the action of these
isomorphisms on sections of the covering
$U_{E_n}\times_XY_n\rightarrow U_{E_n}$. 

Therefore, consider
\[
\xymatrix{
U_{E_n}\times_XY_n \ar[rr]_{\gamma_{U\cap Y_n}}^{\sim} \ar[d] & & U_{E_n}\times_XY_n \ar@/^/[d]  \\
U_{E_n} \ar[rr]_{\gamma_U}^{\sim} & & U_{E_n} \ar@/^/[u]^\sigma .  \\
 }
\]
$\sigma$ being a section. The image of $\sigma$, $(\mu(\gamma)(U))^{-1}(\sigma )$, is determined
by the universal property of $U_{E_n}\times_XY_n $ using the two maps
\begin{eqnarray*}
U_{E_n} & \stackrel{\sigma\circ\gamma_U}{\longmapsto} &
U_{E_n}\times_XY_n \\
U_{E_n} & \stackrel{\id}{\longmapsto} & U_{E_n}.
\end{eqnarray*}
The resulting map is easily seen to be 
\[
\mu(\gamma)(U)^{-1}(\sigma ) = \gamma_{U\cap
  Y_n}^{-1}\circ\sigma\circ\gamma_U.
\]
Thus,
\[
\mu(\gamma)(U)(\sigma ) = \gamma_{U\cap
  Y_n}\circ\sigma\circ\gamma^{-1}_U.
\]

The equivalence of the two actions is confirmed by the following
composite diagram.
\[
\xymatrix{
U_{E_n}\times_XY_n \ar[r]^{\sim}_{\gamma_{U\cap Y_n}} \ar[d]_{\pi_U} &
U_{E_n}\times_XY_n \ar[rr]^{\pi_{Y_n}} \ar@/^/[d]^{\pi_U} & & Y_n \ar[d]
& Y_n \ar[dl] \ar[l]_{\varphi_n} \\
U_{E_n} \ar[r]^{\sim}_{\gamma_U} \ar[d]_{\pi_{E_n}} & U_{E_n}
\ar@/^/[u]^{\sigma} \ar[r]
\ar[d]^{\pi_{E_n}} & U  \ar[r] \ar[d] & X \ar@/_/[dl] & \\
\Spec{E_n} \ar[r]^{\sim}_{\Spec{\gamma}} & \Spec{E_n} \ar[r] &
\Spec{E} \ar@/_/[ur]_{e} & & }
\]
Of particular importance is the commmuting square
\[
\xymatrix{
U_{E_n}\times_XY_n \ar[r]^{\pi_{Y_n}} \ar@/^/[d] & Y_n
\ar[d]^{\varphi_n} \\
U_{E_n} \ar@/^/[u]^{\sigma} \ar[r]_{e_n\circ\pi_{E_n}} & Y_n, }
\]
which implies that 
\begin{equation}\label{eqn:fromsquare}
\pi_{Y_n}\circ\sigma = \varphi_n\circ e_n\circ\pi_{E_n}.
\end{equation}
Once again, the Galois sheaf action is given by 
\begin{eqnarray}\label{eqn:sheafactiondefn}
\gamma\cdot\sigma & = & \mu (\gamma )(U)(\sigma ) \\ \notag & = &
\gamma_{U\cap Y_n}^{-1}\circ\sigma\circ\gamma_U,
\end{eqnarray}
which we can more precisely characterize by taking the projections to
the two components of the fibre product:
\begin{align*}
\pi_{U_{E_n}}\circ (\gamma\cdot\sigma ) & =  \pi_{U_{E_n}}\circ
\gamma_{U\cap Y_n}^{-1}\circ\sigma\circ\gamma_U &\textrm{(by \ref{eqn:sheafactiondefn})}\\
& =  \gamma_U^{-1}\circ\pi_{U_{E_n}}\circ\sigma\circ\gamma_U^{-1} &\\
& =  \id_{U_{E_n}}. &
\end{align*}
\begin{align*}
\pi_{Y_n}\circ (\gamma\cdot\sigma )  & =  \pi_{Y_n}\circ\gamma_{U\cap Y_n}^{-1}\circ\sigma\circ\gamma_U &\textrm{(by \ref{eqn:sheafactiondefn})}\\
& =  \id_{Y_n}\circ\pi_{Y_n}\circ\sigma\circ\gamma_U^{-1} &\textrm{(by
  def'n of $\gamma_{U\cap Y_n}$)}\\
& =  \varphi_n\circ e_n\circ\pi_{E_n}\circ\gamma_U &\textrm{(by~\ref{eqn:fromsquare})}\\
& =  \varphi_n\circ e_n\circ\Spec{\gamma}\circ\pi_{E_n} &\\
& =  \varphi_n\circ\varphi_n\circ e_n\circ\pi_{E_n} & \textrm{(by def'n
  of $\varphi_n$)}\\
& =  \varphi_n\circ\pi_{Y_n}\circ\sigma & \textrm{(by~\ref{eqn:fromsquare})}\\
& =  \pi_{Y_n}\circ (\id_{U_{E_n}}\times\varphi_n)\circ\sigma & \\
& =  \pi_{Y_n}\circ \varphi_n^U\circ\sigma &\\
& =  \pi_{Y_n}\circ (\gamma\cdot\sigma ), &\textrm{(by~\ref{eqn:fgGaloisaction})}
\end{align*}
where the action in the final line refers to the action of the Galois
group induced by the fundamental group action described above in
Equation~\ref{eqn:fgGaloisaction}. 
The Galois sheaf action is therefore trivial on the contribution from
$U_{E_n}$ and is identical to the Galois action from the fundamental
group on the contribution from $Y_n$. We have therefore proved that a finite sheaf on a scheme
over a field $E$ with an action of the absolute Galois group on its
stalk at a chosen geometric point is equivalent to a Galois sheaf,
\ie, the data of a sheaf on a scheme over the algebraic closure of $E$
together with a family of isomorphisms as in
Definition~\ref{def:compatibleaction}. 
\end{proof}

\begin{proposition}\label{cor:galoislocalsystem}
Let $X$ be a scheme over $K$ and suppose $X$ has a $K$-rational point
$x$ with corresponding geometric point
$\cl{x}$. Let $\algfg{X,\cl{x}} =
\algfg{\cl{X},\cl{x}}_{\Gal{\cl{K}/K}}\times\Gal{\cl{K}/K}$ be the
splitting of $\algfg{X,\cl{x}}^\textrm{ab}$ coming from $\pi_1(\cl{x})$.
\begin{enumerate}
\item Let $\mathcal{F}$ be a $\cl{\QQ}_\ell$-local system equivalent
  to
\[
\rho\colon\algfg{X,\cl{x}}\longrightarrow \cl{\QQ}_\ell^\times .
\]
Then, the natural compatible, continuous action of $\Gal{\cl{K}/K}$
arising from base change is equal to the action of $\Gal{\cl{K}/K}$
induced by the action of $\algfg{X,\cl{x}}$ and the Galois character
\[
\Gal{\cl{K}/K}\stackrel{\pi_1(x)}{\longrightarrow}\algfg{X,\cl{x}}\stackrel{\rho}{\longrightarrow}
\cl{\QQ}_\ell^\times .
\]
\item The category of $\cl{\QQ}_\ell$-local systems on $\cl{X}$
  equivalent to
  characters that factor through
  $\algfg{\cl{X},\cl{x}}_{\Gal{\cl{K}/K}}$, equipped with a
  continuous action of $\Gal{\cl{K}/K}$, is equivalent to
  the category of $\cl{\QQ}_\ell$-local systems on $X$.
\end{enumerate}
\end{proposition}
\begin{proof}
Follows from Lemma~\ref{lemma:galoisactionexchange}, together with the homotopy exact
sequence and Proposition~\ref{prop:delignegaloissheaf}, applied term
by term to the component sheaves of $\mathcal{F}$.
\end{proof}




   \chapter{Some Local Systems on the Unipotent Cone in SL(2)}\label{ch:somelocalsystems}

This chapter is devoted to introducing the pair of local systems on
the nilpotent cone in $\SL{2}{K}$ whose nearby cycles will be studied
in the next chapter, and to which we will associate distributions in
Chapter~\ref{ch:distributions}. In the final section, we also
establish that the local systems we have defined are equivariant with
respect to the conjugation action of $\SL{2}{K}$ on itself.

Let $\mathcal{O}$ be a Henselian discrete valuation ring
with field of fractions $K$ and residue field $k$ of characteristic
$p>0$. Assume also that the characteristic
of $K$ is $0$. 

\section{$\SL{2}{K}$ and the Unipotent Cone}\label{section:SLandU}

We begin by describing the setting for the work done in the remainder of
the thesis, the subvariety of regular unipotent elements
$\mathcal{U}_{0,K}$ of the linear algebraic group $\SL{2}{K}$. Our first
step is to recall these schemes and to determine $\algfg{\mathcal{U}_{0,K},\cl{u}}$,
which will allow us to define the local systems.

\subsection{The Schemes}\label{ssection:schemes}


In Section~\ref{section:lsonU}, we will define objects on the linear algebraic group
\[
\operatorname{SL}(2) \ceq \Spec{\ZZ [a,b,c,d]/(ad-bc-1)},
\]
a closed subgroup of $\operatorname{GL}(2) =
\Spec{\ZZ [a,b,c,d]_{(ad-bc)}}$\footnote{The notation here, which appears
  again several times below and on first blush might appear to mean
  localization at a prime ideal, actually means localization of the ring
  $K[a,b,c,d]$ at the multiplicative set composed of powers of
  $ad-bc$. Indeed, the ideal $(ad-bc)$ is not even prime.}. 
To abbreviate notation, we define
\[
\ZZ [\operatorname{SL}(2)] \ceq \ZZ [a,b,c,d]/(ad-bc-1).
\]  
The group structure of $\operatorname{SL}(2)$ is defined by the following
comultiplication, antipode, and identity maps.
\begin{enumerate}
\item Comultiplication: \begin{eqnarray*} m\colon \ZZ [\operatorname{SL}(2)]
    & \longrightarrow & \ZZ [\operatorname{SL}(2)] \otimes_\ZZ \ZZ [\operatorname{SL}(2)] \\
a & \longmapsto & a\otimes a + b \otimes c \\
b & \longmapsto & a\otimes b + b\otimes d \\
c & \longmapsto & a\otimes c + c\otimes d \\
d & \longmapsto & b\otimes c + d\otimes d \end{eqnarray*}
\item Antipode: \begin{eqnarray*} \iota\colon \ZZ [\operatorname{SL}(2)]
    &\longrightarrow & \ZZ [\operatorname{SL}(2)] \\
a & \longmapsto & d \\
b & \longmapsto & -b \\
c & \longmapsto & -c \\
d & \longmapsto & a \end{eqnarray*}
\item Identity: \begin{eqnarray*} \ZZ [\operatorname{SL}(2)] &
    \longrightarrow & \ZZ \\
a &\longmapsto & 1 \\
b & \longmapsto & 0 \\
c & \longmapsto & 0 \\
d & \longmapsto & 1 \end{eqnarray*}
\end{enumerate} 
The familiar matrix group laws are recognizable in these maps and we
will write the operations in their more familiar form (\ie , on
points) when convenient. From now on, write $G=\SL{2}{K} \ceq
\operatorname{SL}(2)\times_{\Spec{\ZZ}}\Spec{K}$. For the coordinate
ring of $G$, write $K[\SL{2}{}]$

The \emph{unipotent} elements $g\in \operatorname{SL}(2)$ are those satisfying
\[
(g-I)^n = 0
\]
for some $n$. As a vanishing condition, this defines a closed
subvariety of $\operatorname{SL}(2)$ for which we can find an alternate condition
by simply considering the characteristic polynomial of such an
element,
\begin{eqnarray*}
\operatorname{\chi}_g(x) &=& x^2 - \operatorname{tr}(g)x +
\operatorname{det}(g) \\
&=& x^2 - \operatorname{tr}(g)x + 1.
\end{eqnarray*}
The minimal polynomial $\operatorname{min}_g(x)$ of $g$ divides this, and when it is a product
of distinct linear factors, $g$ is diagonalizable
and thus not unipotent (the identity being the obvious exception). The only possibilities remaining are
$\operatorname{min}_g(x) = (x+1)^2$ or $\operatorname{min}_g(x) =
(x-1)^2$, the latter of which clearly corresponds to unipotent
$g$. Therefore, the unipotent elements can alternately be
characterized as elements with trace 2, and we can define the
unipotent subvariety as the affine scheme
\[
\mathcal{U} \ceq \Spec{\ZZ [a,b,c,d]/(ad-bc-1, a+d-2)}.
\]
Once again, we will work with $\mathcal{U}_K =
\mathcal{U}\times_{\Spec{\ZZ}}\Spec{K}$ and will simplify notation by setting 
\[
\ku \ceq K[a,b,c,d]/(ad-bc-1,a+d-2).
\]

The identity matrix is a singular point of this variety, and we will need
to work on the smooth subvariety of regular unipotent elements obtained by removing this one
singular (non-regular) point. The resulting scheme is no longer affine, but
can be constructed by gluing together two schemes\footnote{For details on gluing schemes, see
  Hartshorne~\cite{hartshorne}.}. The first is the
open subscheme of $\mathcal{U}_K$ defined by the condition $b\neq 0$,
\[
\mathcal{U}_{b\neq 0} \ceq
\Spec{\ZZ [\mathcal{U}]_{b}},
\]
and the second by $c\neq 0$,
\[
\mathcal{U}_{c\neq 0} \ceq
\Spec{\ZZ [\mathcal{U}]_{c}}.
\]
Both of these schemes contain the open subscheme where $b\neq 0 \neq
c$,
\[
\mathcal{U}_{b\neq 0, c\neq 0} \ceq
\Spec{\ZZ [\mathcal{U}]_{bc}}.\footnote{Once again, the rings in these
  definitions are localizations at the \emph{elements} in the
  subscript, not at prime ideals}
\]
The variety of regular (\ie , nonsingular) unipotent elements $\mathcal{U}_{0}$ is obtained by gluing the
schemes $\mathcal{U}_{b\neq 0}$ and $\mathcal{U}_{c\neq 0}$ along the
open subscheme $\mathcal{U}_{b\neq 0,c\neq 0}$ by the identity
isomorphism. Throughout the chapter, we will work with
$\mathcal{U}_{0,K} \ceq \mathcal{U}_0\times\Spec{K}$.

\subsection{Fundamental Groups}\label{ssection:fundamentalgroups}

In the next section, we will define two $\cl{\QQ}_\ell$-local systems on
$\mathcal{U}_{0,K}$ and extend these local systems to all of
$G = \SL{2}{K}$. By Proposition~\ref{prop:ladiclscorr}
it is possible to define such objects simply by specifying a
continuous character of the \'etale fundamental group
$\algfg{\mathcal{U}_{0,K},\cl{u}}$. Therefore we need to determine this group. 

First, fix a geometric point $\cl{u}$ of $\mathcal{U}_{0,K}$: Let
\begin{eqnarray*}
\cl{u}\colon K[\mathcal{U}]_{(a-1)},K[\mathcal{U}]_{(d-1)} & \longrightarrow & \cl{K} \\
a &\longmapsto & 1 \\
b &\longmapsto & 1 \\
c & \longmapsto & 0 \\
d & \longmapsto & 1;
\end{eqnarray*}
\ie , 
\[
\cl{u} = \left(\begin{array}{cc} 1 & 1 \\ 0 &
    1 \end{array}\right).
\]
This map actually factors through the inclusion
$K\hookrightarrow\cl{K}$, giving a $K$-rational point $u$, which splits the homotopy exact sequence from
Section~\ref{subsection:homexactseq} applied to
$\mathcal{U}_{0,K}$:
\[
1\rightarrow \algfg{\cl{\mathcal{U}}_0,\cl{u}}\rightarrow
\algfg{\mathcal{U}_{0,K},\cl{u}}\rightarrow \Gal{\cl{K}/K}\rightarrow 1,
\]
where $\cl{\mathcal{U}}_0 \ceq
  \mathcal{U}_{0,K}\times_{\Spec{K}}\Spec{\cl{K}}$. Thus, we only need to determine the
  fundamental group of $\cl{\mathcal{U}}_0$ to calculate
  $\algfg{\mathcal{U}_{0,K},\cl{u}}$.
\begin{lemma}\label{lemma:geofgU}
$\algfg{\cl{\mathcal{U}}_0,\cl{u}} \iso \mu_2$.
\end{lemma}
\begin{proof}
For this we deploy the comparison results
Proposition~\ref{prop:topfgcomparison} and
Proposition~\ref{prop:algclosedcomparisonfg}, which tell us that
$\algfg{\cl{\mathcal{U}}_0,\cl{u}} \iso
\widehat{\operatorname{\pi_1^\mathrm{top}}(\mathcal{U}_{0,\CC})}$. As
a complex manifold $\mathcal{U}_{0,\CC}$ has a universal cover, the
punctured affine plane $\mathbb{A}^2_{0,\CC}$:
\begin{align*}
\mathbb{A}^2_{0,\CC} &\longrightarrow  \mathcal{U}_{0,\CC} \\
(x,y) &\longmapsto \begin{pmatrix} 1-xy & x^2 \\ -y^2 & 1 + xy \end{pmatrix}.
\end{align*}

That $\mathbb{A}^2_{0,\CC}$ is simply connected follows from the fact
that it is homotopy equivalent to $S^1_\CC \iso S^2_\RR$, which can be
shown to be simply connected by invoking the Seifert-Van Kampen
theorem\footnote{See, for exmaple, Munkres~\cite[Theorem 70.1]{munkres}.}. Hence,
\[
\operatorname{\pi_1^\mathrm{top}}(\mathcal{U}_{0,\CC}) \iso
\Aut_{\mathcal{U}_{0,\CC}}(\mathbb{A}^2_{0,\CC}),
\]
and this latter group consists of one nontrivial automorphism that
maps $x$ to $-x$ and $y$ to $-y$. Thus, by
Propositions~\ref{prop:topfgcomparison}
and~\ref{prop:algclosedcomparisonfg},
\[
\algfg{\cl{\mathcal{U}}_0,\cl{u}} \iso
\widehat{\operatorname{\pi_1^\mathrm{top}}(\mathcal{U}_{0,\CC})} =
\widehat{\Aut_{\mathcal{U}_{0,\CC}}(\mathbb{A}^2_{0,\CC})} =
\widehat{\mu_2} = \mu_2.
\]
\end{proof}



\begin{lemma}\label{lemma:algfgU}
$\algfg{\mathcal{U}_{0,K},\cl{u}}\iso \mu_2\times\Gal{\cl{K}/K}$.
\end{lemma}
\begin{proof}
The $K$-rational point $u$
splits the homotopy exact sequence, which, by
Lemma \ref{lemma:geofgU}, yields
\[
\algfg{\mathcal{U}_{0,K},\cl{u}}\iso \mu_2\rtimes\Gal{\cl{K}/K}.
\]
To see that the conjugation action of $\Gal{\cl{K}/K}$ is trivial,
consider an arbitrary connected Galois cover $Y$ of
$\mathcal{U}_{0,K}$. The automorphism group of $Y$ over $\mathcal{U}_{0,K}$
contains at most one nontrivial ``geometric'' automorphism: the image
of $-1$ under the functorial map from
$\algfg{\cl{\mathcal{U}}_0,\cl{u}}$. As is shown in the proof of
Lemma~\ref{lemma:defnf} below, this automorphism multiplies
coordinates by $-1$. The image of
$\gamma\in\Gal{\cl{K}/K}$ under $\pi_1(\cl{u})$ in the automorphism group of
$Y$ will either be the automorphism of $Y$ induced by $\gamma$ (which
may be trivial), or that automorphism composed with the geometric
automorphism. Conjugation of the geometric morphism by either of these will have no effect, since the geometric morphism is,
like $-1$, of order 2, and the natural induced Galois morphisms of $Y$
commute with the geometric automorphism. Thus, the semidirect product above is in fact
a direct product.
\end{proof}

\section{Local Systems on $\mathcal{U}_{0,K}$}\label{section:lsonU}

With the precise form of $\algfg{\mathcal{U}_{0,K},\cl{u}}$ now in hand,
the next step is to define the local systems we will study on
$\mathcal{U}_{0,K}$. Let $\varpi$ be a uniformizer
of the ring of integers of $K$, $\mathcal{O}$, and define the morphism of $K$-varieties $m(\varpi
)\colon G\rightarrow G$ by
\begin{align*}
m(\varpi )\colon K[\SL{2}{}] & \longrightarrow K[\SL{2}{}] \\
a &\longmapsto a \\
b &\longmapsto \frac{b}{\varpi} \\
c &\longmapsto \varpi c \\
d & \longmapsto d .
\end{align*}
On points, this is given by
\begin{align*}
\begin{pmatrix} a & b \\ c & d \end{pmatrix} &
\longmapsto \begin{pmatrix} a & \frac{b}{\varpi} \\ \varpi c &
  d \end{pmatrix} 
\end{align*}
\begin{definition}\label{def:thelocalsystems}
Let
\[
\rho_{\oE}\colon\algfg{\mathcal{U}_{0,K},\cl{u}}\longrightarrow
\cl{\QQ}_\ell
\]
be the product of the nontrivial character of $\mu_2$ and the
trivial character of $\Gal{\cl{K}/K}$, with respect to the product
decomposition $\algfg{\mathcal{U}_{0,K},\cl{u}}\iso
\mu_2\times\Gal{\cl{K}/K}$ defined by the choice of $K$-rational point
\[
u = \left( \begin{array}{cc} 1&1\\0&1\end{array}\right).
\] 
Let $\oE$ be the local system
on $\mathcal{U}_{0,K}$ corresponding to $\rho_{\oE}$ under the equivalence
of Proposition~\ref{prop:ladiclscorr}. In addition, let
\[
\oE^\prime \ceq m(\varpi )^*(\oE ).
\]
\end{definition}
Next, we study some alternate ways of defining $\oE$ and $\oE^\prime$
that will aid in the proof of Theorem~\ref{prop:mainlocalsystems}. In
particular, we are interested in defining the local systems through
characters of the automorphism groups of specific Galois covers of $\mathcal{U}_{0,K}$.

These covers will be the punctured affine plane $\mathbb{A}^2_{0,K}$. Like
$\mathcal{U}_{0,K}$, this is a quasiaffine scheme constructed by
gluing the two copies of $\mathbb{A}^2$, one with the $x$-axis deleted, and
second with the $y$-axis deleted,
\[
\mathbb{A}^2_{y\neq 0} \ceq \Spec{\ZZ[x,y]_y} 
\]
and
\[
\mathbb{A}^2_{x\neq 0} \ceq \Spec{\ZZ[x,y]_x}
\]
along the common open subscheme
\[
\mathbb{A}^2_{x\neq 0\neq y} \ceq \Spec{\ZZ[x,y]_{xy}}
\]
with the identity isomorphism.
\begin{remark}\label{rem:affineforquasiaffine}
Just as $\mathcal{U}_0$ and $\mathbb{A}^2_0$ are contained in
$\mathcal{U}$ and $\mathbb{A}^2$ as open subschemes, both
$\mathcal{U}_{0,K}$ and $\mathbb{A}^2_{0,K}$ are contained as open
subvarieties in the affine varieties $\mathcal{U}_K$ and $\mathbb{A}^2_K$,
and all of the Galois covers we will consider can be defined on these
affine schemes and obtained by restriction to $\mathbb{A}^2_{0,K}$. This
is the approach we will take because it is much cleaner than defining
maps on components of the gluing data for $\mathcal{U}_{0,K}$ and
$\mathbb{A}^2_{0,K}$. But it should be noted that the affine versions of
these covers are not \'etale, as they all ramify at the identity in
$\mathcal{U}_K$. Thus, the notions and methods of
Chapter~\ref{ch:ladicsheaves} are not accessible for these maps;
restricting attention to the open subvarieties is a necessary step.
\end{remark}
\begin{lemma}\label{lemma:defnf}
The character $\rho_{\oE}$ is obtained by inflating the nontrivial
character of the degree 2 Galois cover given on points by
\begin{align*}
f\colon \mathbb{A}^2_{0,K} &\longrightarrow \mathcal{U}_{0,K} \\
(x,y) & \longmapsto \begin{pmatrix} 1 - xy & x^2 \\ -y^2 & 1 + xy \end{pmatrix}
\end{align*}
and defined on coordinate rings by
\begin{align*}
K[\mathcal{U}] &\longrightarrow K[\mathbb{A}^2] \\
a&\longmapsto 1 -xy \\
b&\longmapsto x^2 \\
c&\longmapsto -y^2 \\
d&\longmapsto 1+xy,
\end{align*}
to all of $\algfg{\mathcal{U}_{0,K},\cl{u}}$.
\end{lemma}
\begin{proof}
That $f$ is really an \'etale morphism is confirmed using the Jacobean
criterion for \'etaleness (see, for example, Freitag and Kiehl~\cite{freitagkiehl},
Chapter I, Prop. 1.3).

The rest of the lemma follows from the fact that the functorial map
\[
\pi_1(\operatorname{pr})\colon\algfg{\cl{\mathcal{U}},\cl{u}}\longrightarrow
\algfg{\mathcal{U}_{0,K},\cl{u}}
\]
coming from the canonical projection
$\operatorname{pr}\colon\cl{\mathcal{U}}_0\rightarrow \mathcal{U}_{0,K}$
induces an isomorphism between $\algfg{\cl{\mathcal{U}}_0,\cl{u}}$ and
$\Aut_{\mathcal{U}_{0,K}}(f)$.

To see this, recall Equation~\ref{eqn:functorialrelation} and the
discussion surrounding it, which explained
how the functorial map is defined. In particular, if $\pr_f$ is the
canonical projection $\algfg{\mathcal{U}_{0,K},\cl{u}}\rightarrow
\Aut_{\mathcal{U}_{0,K}}(f)$, then the composition
$\operatorname{pr}_f\circ\pi_1(\operatorname{pr})$ is defined by
pulling back the diagram
\[
\xymatrix{
& \mathbb{A}^2_{0,K} \ar[d]^f \\
\cl{\mathcal{U}}_0 \ar[r]_{\operatorname{pr}} & \mathcal{U}_{0,K} }
\]
to give
\[
\xymatrix{
\mathbb{A}^2_{0, \cl{K}} \ar[d]_{\cl{f}} \ar[r]^{\operatorname{pr}^{\mathbb{A}^2}} & \mathbb{A}^2_{0,K} \ar[d]^f \\
\cl{\mathcal{U}}_0 \ar[r]_{\operatorname{pr}} & \mathcal{U}_{0,K} .}
\]
Then, $\pi_1(\pr )$ is defined by the relation
\[
(\operatorname{pr}^{\mathbb{A}^2})\circ\phi =
\pi_1(\operatorname{pr})(\phi )\circ (\operatorname{pr}^{\mathbb{A}^2}).
\]
The fibres of $f$ at all points aside from the identity contain 2
points, and its automorphism group consists of one nontrivial map,
denoted $(-1)$, that sends $x$ to $-x$ and $y$ to $-y$. \'Etale covers
are stable under base change, hence $\cl{f}$ is a cover as well. It is
likewise degree 2 with a single nontrivial covering map defined the
same way as $(-1)$, which we distinguish by denoting it by
$\cl{(-1)}$. This cover is the universal cover of $\cl{\mathcal{U}}_0$
implied by the proof of Lemma~\ref{lemma:geofgU}---the re-use of the
notation $\cl{f}$ is no accident. Since $\operatorname{pr}^{\mathbb{A}^2}$ is
equal to the canonical projection
$\mathbb{A}^2_{\cl{K}}\rightarrow\mathbb{A}^2_K$, it is clear that the
map
\[
\algfg{\cl{\mathcal{U}}_0,\cl{u}}\iso\Aut_{\cl{\mathcal{U}}_0}(\cl{f})
\stackrel{\pi_1(\operatorname{pr})}{\longrightarrow}
\Aut_{\mathcal{U}_{0,K}}(f)
\]
is an isomorphism.

To ensure that inflation of the nontrivial character of
$\Aut_{\mathcal{U}_{0,K}}(f)$ to $\algfg{\mathcal{U}_{0,K},\cl{u}}$ gives the
trivial character on $\Gal{\cl{K}/K}$, the next step is to calculate the composition
\[
\Gal{\cl{K}/K}\stackrel{\pi_1(u)}{\longrightarrow}
\algfg{\mathcal{U}_{0,K},\cl{u}}
\stackrel{\operatorname{pr}_f}{\longrightarrow}
\Aut_{\mathcal{U}_{0,K}}(f).
\]
Let $Y$ be any finite Galois cover of $\Spec{K}$. Thus, $Y=\Spec{L}$,
where $L$ is a finite Galois extension of $K$. The relevant Cartesian
diagram to consider is 
\[
\xymatrix{
\mathbb{A}^2_{0,L} \ar[d]_{f_L} & \Spec{L} \ar[l] \ar[d] \\
\mathcal{U}_{0,K} & \Spec{K} \ar[l]^{\cl{u}} }
\]
where $f_L$ is defined analogously to $f$. The covering group of $f_L$
consists of the automorphism $(-1)_L$, defined analogously to
$(-1)\in\Aut_{\mathcal{U}_{0,K}}(f)$, and the automorphisms induced from
elements of $\Gal{L/K}$. Thus, 
\[
\Aut_{\mathcal{U}_{0,K}}(f_L)\iso
\Aut_{\mathcal{U}_{0,K}}(f)\times\Gal{L/K}\footnote{Although it turns out
  to be true in this case, it is not necessarily true that this
  decomposition exactly reflects the splitting we have for
  the fundamental group. Compare, for example, what happens for the
  map $f^\prime$ in the proof of Lemma~\ref{lemma:defnfprime}.} 
\]
and
\[
\Aut_{\Spec{K}}(\Spec{L})\iso \Gal{L/K}.
\]
The top horizontal arrow in
the Cartesian square is the morphism that factors $\cl{u}$ through $\Spec{L}$; on coordinates, it is
defined the same way as $\cl{u}$. This implies that
\[
\Gal{L/K}\stackrel{(\operatorname{pr}_{f_L})\circ\pi_1(u)}{\longrightarrow}
\Aut_{\mathcal{U}_{0,K}}(f)\times \Gal{L/K}
\]
is an isomorphism onto the second term.

The final step is to calculate the maps
\[
g\colon\Aut_{\mathcal{U}_{0,K}}(f)\longrightarrow
\Aut_{\mathcal{U}_{0,K}}(f_L)
\]
induced by the map of covers $\mathbb{A}^2_{0,L}\rightarrow \mathbb{A}^2_{0,K}$
present for every finite Galois extension $L$ of $K$. For this, we
consider the commuting triangle
\[
\xymatrix{
\mathbb{A}^2_{0,K} \ar[rr]^\tau \ar[dr]_{f} & & \mathbb{A}^2_{0,L}
\ar[dl]^{f_L} \\ 
& \mathcal{U}_{0,K} .& }
\]
For $\phi\in\Aut_{\mathcal{U}_{0,K}}(f)$, $g(\phi )$ is defined by the similar relation
\[
\tau\circ\phi = g(\phi )\circ\tau ,
\]
from which we see that
\[
g((-1)) = (-1)_L.
\]
Together, these last two calculations show that the nontrivial character of
$\Aut_{\mathcal{U}_{0,K}}(f)$ inflates to $\rho_{\oE}$ on $\algfg{\mathcal{U}_{0,K},\cl{u}}$.
\end{proof}
In what follows, we will need to make reference to a square root of
the uniformizer $\varpi$. Therefore, set $K^\prime \ceq
K(\sqrt{\varpi}) = K[x]/(x^2-\varpi )$.
\begin{lemma}\label{lemma:defnfprime}
The local system $\oE^\prime$ is equivalent to the $\ell$-adic
character of $\algfg{\mathcal{U}_{0,K},\cl{u}}$ obtained by inflating the
nontrivial character of the automorphism group of the degree 2 \'etale
cover
\begin{align*}
f^\prime\colon \mathbb{A}^2_{0,K} &\longrightarrow \mathcal{U}_{0,K}
\\
(x,y) & \longmapsto \begin{pmatrix} 1 - xy & \frac{x^2}{\varpi} \\
  -\varpi y^2 & 1+xy \end{pmatrix}.
\end{align*}
This cover is given on coordinates by
\begin{align*}
K[\mathcal{U}] &\longrightarrow K[\mathbb{A}^2] \\
a&\longmapsto 1 -xy \\
b&\longmapsto \frac{x^2}{\varpi} \\
c&\longmapsto -\varpi y^2 \\
d&\longmapsto 1+xy.
\end{align*}
\end{lemma}
\begin{proof}
In view of Lemma~\ref{lemma:defnf}, we can find a fundamental group
character equivalent to $\oE^\prime$ by base changing the cover $f$ by
the automorphism $m(\varpi )$. The necessary automorphism in the
commuting square
\[
\xymatrix{
\mathbb{A}^2_{0,K} \ar[d]_f \ar[r] & \mathbb{A}^2_{0,K} \ar[d]^f \\
\mathcal{U}_{0,K} \ar[r]_{m(\varpi )} & \mathcal{U}_{0,K} }
\]
sends $x$ to $\frac{x}{\sqrt{\varpi}}$ and $y$ to
$\sqrt{\varpi}y$---it is not defined over $K$, nor over an \'etale
extension of $K$. The Cartesian
square to consider instead is
\[
\xymatrix{
\mathbb{A}^2_{0,K} \ar[d]_{f^\prime} \ar@{=}[r] & \mathbb{A}^2_{0,K}
\ar[d]^f \\
\mathcal{U}_{0,K} \ar[r]_{m(\varpi )} & \mathcal{U}_{0,K} ,}
\]
with $f^\prime$ defined as in the statement of the lemma. Like $f$,
$f^\prime$ is a degree 2 \'etale cover, and
\[
\Aut_{\mathcal{U}_{0,K}}(f^\prime )= \{\id ,(-1)^\prime \},
\]
where $(-1)^\prime$ is defined like $(-1)$. The map of automorphism
groups is an isomorphism, and the character of
$\Aut_{\mathcal{U}_{0,K}}(f^\prime )$ obtained by composition is the
nontrivial one, as claimed.
\end{proof}
\begin{lemma}\label{lemma:rhoEprime}
Let $\algfg{\mathcal{U}_{0,K},\cl{u}}\iso\mu_2\times\Gal{\cl{K}/K}$ be the
decomposition induced by the choice of the $K$-rational point
$\cl{u}$. Then, under the equivalence of
Proposition~\ref{prop:ladiclscorr}, $\oE^\prime$ corresponds to the
character
\[
\rho_{\oE^\prime}\colon\mu_2\times\Gal{\cl{K}/K}\longrightarrow
\cl{\QQ}_\ell
\]
satisfying
\begin{align*}
\rho_{\oE^\prime}|_{\mu_2} & = \operatorname{sgn} \\
\rho_{\oE^\prime}|_{\Gal{\cl{K}/K}}(\gamma ) & = 
\begin{cases} 1 & \gamma |_{K^\prime}=\id_{K^\prime} \\ -1 &
      \gamma |_{K^\prime} = \conj .  \end{cases}
\end{align*}
\end{lemma}
\begin{proof}

First, we calculate the composite map
\[
\algfg{\cl{\mathcal{U}}_0,\cl{u}}\stackrel{\pi_1(\operatorname{pr})}{\longrightarrow}
\algfg{\mathcal{U}_{0,K},\cl{u}}
\stackrel{\operatorname{pr}_{f^\prime}}{\longrightarrow}
\Aut_{\mathcal{U}_{0,K}}(f^\prime )
\]
by looking at the Cartesian square
\[
\xymatrix{
\mathbb{A}^2_{0,\cl{K}} \ar[d]_{\cl{f}} \ar[r] & \mathbb{A}^2_{0,K}
\ar[d]^{f^\prime} \\
\cl{\mathcal{U}}_0 \ar[r]_{\operatorname{pr}} & \mathcal{U}_{0,K}.}
\]
The map between the covering spaces is defined by
\begin{align*}
x &\longmapsto \frac{x}{\sqrt{\varpi}} \\
y &\longmapsto \sqrt{\varpi}y.
\end{align*} 
Therefore, the composite map is an isomorphism of groups. For the other half of the decomposition, we calculate
\[
\Gal{L/K}\stackrel{\pi_1(u)}{\longrightarrow}
\algfg{\mathcal{U}_{0,K},\cl{u}}
\stackrel{\pr_{f^\prime}}{\longrightarrow}
\Aut_{\mathcal{U}_{0,K}}(f^\prime ).
\]
As in the proof of Lemma~\ref{lemma:defnf}, for any finite Galois
extension $L$ of $K$, we have the Cartesian square
\[
\xymatrix{
\mathbb{A}^2_{0,L} \ar[d]_{f^\prime_L} & \Spec{L} \ar[d] \ar[l] \\
\mathcal{U}_{0,K} & \Spec{K} \ar[l]^{u}.}
\]
The map $\Spec{L}\rightarrow \mathbb{A}^2_{0,L}$ arises from the ring
homomorphism
\begin{align*}
L[\mathbb{A}^2] &\longrightarrow L \\
x&\longmapsto \sqrt{\varpi} \\
y&\longmapsto 0
\end{align*}
(so in fact, $L$ must be an extension of $K$ containing
$\sqrt{\varpi}$). Therefore, taking 
\[ \Aut_{\mathcal{U}_{0,K}}(f_L^\prime )
\iso \mu_2\times \Gal{L/K},
\]
we obtain
\begin{align*}
(\pr_f\circ\pi_1(\cl{u}))(\gamma ) & = \begin{cases} \gamma
      , & \gamma |_{K^\prime}=\id_{K^\prime} \\ (-1)\circ\gamma , &
      \gamma |_{K^\prime} = \Conj. \end{cases}
\end{align*}
From this, it follows that the character of
$\algfg{\mathcal{U}_{0,K},\cl{u}}$ obtained by inflating the nontrivial
character of $\Aut_{\mathcal{U}_{0,K}}(f^\prime )$ is $\rho_{\oE^\prime}$.
\end{proof}
For the purposes of the proof of Theorem~\ref{prop:mainlocalsystems},
it will be useful to characterize $\oE$ and $\oE^\prime$ in terms of
two more Galois covers of $\mathcal{U}_{0,K}$. In both cases, the covering
space is the punctured affine plane $\mathbb{A}^2_{0,K^\prime}$ over
the quadratic extension $K^\prime$ defined above. Once again this is
a nonaffine scheme, but, in tune with
Remark~\ref{rem:affineforquasiaffine}, we will work with it as on open
subscheme of the affine scheme $\mathbb{A}^2_{K^\prime}$.
\begin{theorem}\label{lemma:varpicovers}
The local systems $\oE$ and $\oE^\prime$ are equivalent to the characters
of $\algfg{\mathcal{U}_{0,K},\cl{u}}$ induced by inflation from characters
of the automorphism groups of Galois covers
\[
f_\varpi , f^\prime_\varpi\colon\mathbb{A}^2_{0,K^\prime}\longrightarrow
\mathcal{U}_{0,K}
\]
defined on coordinate rings by
\begin{eqnarray*}
K[\mathcal{U}] &\longrightarrow & K^\prime [x,y] \\
f_\varpi\colon a & \longmapsto & 1 - xy \\
b &\longmapsto & x^2 \\
c & \longmapsto & -y^2 \\
d &\longmapsto & 1+xy
\end{eqnarray*} and
\begin{eqnarray*}
f_\varpi^\prime\colon a & \longmapsto & 1-xy \\
b &\longmapsto & \frac{x^2}{\varpi} \\
c &\longmapsto & \-\varpi y^2 \\
d & \longmapsto & 1+xy
\end{eqnarray*}
and given on points by
\begin{align*}
f_\varpi\colon (x,y) &\longmapsto \begin{pmatrix} 1-xy & x^2 \\ -y^2 &
  1+xy \end{pmatrix}
\end{align*}
and
\begin{align*}
f^\prime_\varpi\colon (x,y) &\longmapsto \begin{pmatrix} 1-xy &
  \frac{x^2}{\varpi} \\ -\varpi y^2 &
  1+xy \end{pmatrix}
\end{align*}
\end{theorem}
\begin{proof}
The precise characters will be given below. Both $f_\varpi$ and $f_\varpi^\prime$
have automorphism groups of order 4, consisting of $\id$, $\varphi_1$, $\varphi_2$,
and $\varphi_3$ defined by
\begin{eqnarray*}
\varphi_1\colon x &\longmapsto & -x \\
y &\longmapsto & -y \\
\varphi_2\colon \sqrt{\varpi} &\longmapsto & -\sqrt{\varpi} \\
x &\longmapsto & x \\
y &\longmapsto & y \\
\varphi_3\colon &  \varphi_2\circ\varphi_1; &
\end{eqnarray*}
\ie , $\varphi_1$ multiplies coordinates by $-1$ and $\varphi_2$ acts
on coefficients according to the nontrivial Galois automorphisms. To
avoid confusion, we will denote the automorphisms of $f_\varpi^\prime$ by
$\varphi^\prime_i$, $i=1,2,3$. To see how the characters corresponding
to our local systems relate to characters of these automorphism groups, we
need to explicitly find the four maps
\[
\xymatrix{
&&& \Aut_{\mathcal{U}_{0,K}}(f_\varpi ) \\
\algfg{\cl{\mathcal{U}}_0,\cl{u}} \ar[rr]^{\pi_1(\pr )} & &
\algfg{\mathcal{U}_{0,K},\cl{u}} \ar[ur]^{\pi_{f_\varpi}} \ar[dr]_{\pi_{f_\varpi^\prime}}
& \\
&&& \Aut_{\mathcal{U}_{0,K}}(f_\varpi^\prime) }
\]
and
\[
\xymatrix{
&&& \Aut_{\mathcal{U}_{0,K}}(f_\varpi ) \\
\Gal{\cl{K}/K} \ar[rr]^{\pi_1(u)} & &
\algfg{\mathcal{U}_{0,K},\cl{u}} \ar[ur]^{\pi_{f_\varpi}} \ar[dr]_{\pi_{f_\varpi^\prime}}
& \\
&&& \Aut_{\mathcal{U}_{0,K}}(f_\varpi^\prime), }
\]
where the leftmost maps are the functorial maps coming from,
respectively, the projection $\cl{\mathcal{U}}_0\rightarrow
\mathcal{U}_{0,K}$ and the $K$-rational point
$\cl{u}\colon\Spec{K}\rightarrow\mathcal{U}_{0,K}$, and the maps $\pi_{f_\varpi}$
and $\pi_{f_\varpi^\prime}$ are the canonical projections.
The discussion around Equation~\ref{eqn:functorialrelation} details how to find the
these compositions. For the first pair, we must consider the commuting diagram
\[
\xymatrix{
\mathbb{A}^2_{0,\cl{K}} \ar[rr]^{\phi}
\ar[drr]^{\cl{f}_\varpi,\cl{f}_\varpi^\prime}  & & \mathbb{A}^2_{0,\cl{K}} \ar[r]^{\beta}
\ar[d] & \mathbb{A}^2_{0,K^\prime} \ar[rr]^{\pi_1(\pi )(\phi )}
\ar[d]^{f_\varpi,f_\varpi^\prime} & &
\mathbb{A}^2_{0,K^\prime} \ar[dll] \\
& & \cl{\mathcal{U}}_{0,K} \ar[r]_{\pi} & \mathcal{U}_{0,K} &}
\]
where the centre square is Cartesian, defined by the right and bottom
sides. On
coordinate rings, the map $\beta$ is the
natural inclusion $K^\prime [x,y]\hookrightarrow \cl{K}[x,y]$ defined by a
choice of embedding $K^\prime \hookrightarrow \cl{K}$ (which is part
of the choice of geometric point for the covers). The maps
$\algfg{\cl{\mathcal{U}}_0,\cl{u}} \rightarrow
\Aut_{\mathcal{U}_{0,K}}(f_\varpi
),\Aut_{\mathcal{U}_{0,K}}(f_\varpi^\prime )$ are then defined
  by the relation
\[
\beta\circ\phi = \pi_1(\pi )(\phi )\circ\beta.
\]
The cover $\mathbb{A}^2_{0,\cl{K}}\rightarrow\cl{\mathcal{U}}_0$ along
the left side of the Cartesian square is defined on coordinates in
the same way as $f_\varpi$ or
$f_\varpi^\prime$, depending on which lies on the right side of the square, and in either
case has automorphism group of order 2, consisting of the identity and
nontrivial map $(-1)$, which sends $x$ to $-x$ and $y$ to $-y$. From this,
we see that the two compositions are given by
\begin{eqnarray*}
\algfg{\cl{\mathcal{U}}_0,\cl{u}} & \longrightarrow & \Aut_{\mathcal{U}_{0,K}}(f_\varpi ) \\
\id & \longmapsto & \id \\
(-1) & \longmapsto & \varphi_1 ; \\
\algfg{\cl{\mathcal{U}}_0,\cl{u}} & \longrightarrow & \Aut_{\mathcal{U}_{0,K}}(f_\varpi^\prime ) \\
\id & \longmapsto & \id \\
(-1) & \longmapsto & \varphi^\prime_1. 
\end{eqnarray*}
For the two compositions coming from the Galois group, recall
the choice of $K$-rational point
$\cl{u}\colon\Spec{K}\rightarrow\mathcal{U}_{0,K}$:
\begin{eqnarray*}
\cl{u}\colon K[\mathcal{U}_{0}] & \longrightarrow & K \\
a & \longmapsto & 1 \\
b & \longmapsto & 1 \\
c & \longmapsto & 0 \\
d & \longmapsto & 1.
\end{eqnarray*}
The compositions of
$\pi_1(\cl{u})\colon\Gal{\cl{K}/K}\rightarrow\algfg{\mathcal{U}_{0,K},\cl{u}}$
with the canonical projections to $\Aut_{\mathcal{U}_{0,K}}(f_\varpi )$ and $\Aut_{\mathcal{U}_{0,K}}(f_\varpi^\prime )$ are defined
by considering the similar commuting diagram
\[
\xymatrix{
\Spec{K^\prime} \ar[r]^{\phi} \ar[dr] & \Spec{K^\prime} \ar[r]^{\ \ \beta}
\ar[d] & \mathbb{A}^2_{0,K^\prime} \ar[rr]^{\pi_1(\cl{u})(\phi )}
\ar[d]^{f_\varpi , f_\varpi^\prime} & &
\mathbb{A}^2_{0,K^\prime} \ar[dll] \\
&  \Spec{K} \ar[r] & \mathcal{U}_{0,K} &}
\]
where again the middle square is Cartesian and the morphism $\beta$ is
the morphism determined by the choice of geometric point for the
cover $\Spec{K^\prime}$. With $f_\varpi$ along the right side, $\beta$
is defined by $(x,y)\mapsto (1,0)$; with $f^\prime_\varpi$, $\beta$
maps $(x,y)$ to $(\sqrt{\varpi},0)$. $\Spec{K^\prime}$ has
two covering automorphisms, the identity and a conjugation morphism defined by
$\sqrt{\varpi} \mapsto -\sqrt{\varpi}$. In this case we find that the composite maps are 
\begin{eqnarray*}
\Gal{\cl{K}/K} & \longrightarrow & \Aut_{\mathcal{U}_{0,K}}(f_\varpi ) \\
\id & \longmapsto & \id \\
\operatorname{conj} & \longmapsto & \varphi_2 \\
\Gal{\cl{K}/K} &\longrightarrow & \Aut_{\mathcal{U}_{0,K}}(f_\varpi^\prime ) \\
\id & \longmapsto & \id \\
\operatorname{conj} & \longmapsto & \varphi^\prime_3.
\end{eqnarray*}
This information together with the lemmata above imply that $\oE$ and
$\oE^\prime$ are equivalent to fundamental group characters induced by
the following characters of the automorphism groups of $f_\varpi$ and $f_\varpi^\prime$.
\begin{eqnarray*}
\oE\colon \Aut_{\mathcal{U}_{0,K}}(f_\varpi)&\longrightarrow&\cl{\QQ}^\times_\ell \\
\id & \longmapsto & 1 \\
\varphi_1 & \longmapsto & -1 \\
\varphi_2 & \longmapsto & 1 \\
\varphi_3 & \longmapsto & -1 
\end{eqnarray*}
\begin{eqnarray*}
 \Aut_{\mathcal{U}_{0,K}}(f_\varpi^\prime ) &\longrightarrow&\cl{\QQ}^\times_\ell \\
\id & \longmapsto & 1 \\
\varphi_1 & \longmapsto & -1 \\
\varphi_2 & \longmapsto & -1 \\
\varphi_3 & \longmapsto & 1
\end{eqnarray*}
\begin{eqnarray*}
\oE^\prime\colon \Aut_{\mathcal{U}_{0,K}}(f_\varpi )&\longrightarrow&\cl{\QQ}^\times_\ell \\
\id & \longmapsto & 1 \\
\varphi_1 & \longmapsto & -1 \\
\varphi_2 & \longmapsto & -1 \\
\varphi_3 & \longmapsto & 1 
\end{eqnarray*}
\begin{eqnarray*}
\Aut_{\mathcal{U}_{0,K}}(f_\varpi^\prime )&\longrightarrow&\cl{\QQ}^\times_\ell \\
\id & \longmapsto & 1 \\
\varphi_1 & \longmapsto & -1 \\
\varphi_2 & \longmapsto & 1 \\
\varphi_3 & \longmapsto & -1 .
\end{eqnarray*}
\end{proof}
We end this section by depicting in the commutative diagram below the
relationship between all of the Galois covers we have used.
\[
\xymatrix{
\mathbb{A}^2_{0,K^\prime} \ar@{<->}[rr]^{\tau}
\ar[dr]^{f_\varpi^\prime} \ar[dd] & & \mathbb{A}^2_{0,K^\prime}
\ar[dl]_{f_\varpi} \ar[dd] \\
& \mathcal{U}_{0,K} & \\
\mathbb{A}^2_{0,K} \ar[ur]_{f^\prime} & & \mathbb{A}^2_{0,K} \ar[ul]^{f}}
\]
where the arrow $\tau$ is an isomorphism of covers defined by
\begin{align*}
x &\longmapsto  \sqrt{\varpi}x \\
y &\longmapsto  \frac{y}{\sqrt{\varpi}}
\end{align*}
(therefore it is defined only over $K^\prime$, not $K$).
From the proof of Lemma~\ref{lemma:varpicovers} we can deduce that the map $\Aut_{\mathcal{U}_{0,K}}(f_\varpi) \rightarrow
\Aut_{\mathcal{U}_{0,K}}(f_\varpi^\prime )$ induced by the isomorphism of covers $\tau$ is given by
\begin{eqnarray*}
\id &\longmapsto & \id \\
\varphi_1 & \longmapsto & \varphi^\prime_1 \\
\varphi_2 &\longmapsto & \varphi^\prime_3 \\
\varphi_3 &\longmapsto & \varphi^\prime_2.
\end{eqnarray*}
It is this `swapping' of the latter two elements that is really the key to
the result in Theorem~\ref{prop:mainlocalsystems}. 

\section{Equivariance}\label{section:equivariance}

In this section, we prove that the local systems $\oE$ and
$\oE^\prime$ are equivariant with respect to the conjugation action of
$\SL{2}{K}$ on $\mathcal{U}_{0,K}$. All notation established in the
previous sections remains in effect, unless otherwise stated.

As the reader may discover by consulting Bernstein and Lunt's
book~\cite{BernLunts}, the most general notion of equivariant sheaves is rather
complicated. Here we use a simpler definition sufficient for the requirements
of this thesis. 

\begin{definition}\label{def:equivariant}
Let $G$ be a linear algebraic group and $X$ a scheme over a field
$K$. If $\alpha\colon G\times X \rightarrow X$ is an action of $G$ on
$X$ defined over $K$, and $\mathcal{F}$ is a $\cl{\QQ}_\ell$-local system on
$X$, $\mathcal{F}$ is \emph{$G$-equivariant} if there exists an
isomorphism
\[
\alpha^*\mathcal{F} \stackrel{\sim}{\longrightarrow} p^*\mathcal{F},
\]
where $p\colon G\times X\rightarrow X$ is the canonical projection.
\end{definition}

\begin{definition}\label{def:actions}
The conjugation action $\Conj\colon\SL{2}{K}\times\mathcal{U}_K \longrightarrow \mathcal{U}_K$
of $\SL{2}{K}$ on $\mathcal{U}_K$ is defined on coordinate rings by
\begin{eqnarray*}
\Conj\colon K[\mathcal{U}_0] & \longrightarrow & K[\alpha ,\beta ,\gamma ,
\delta]/(\alpha\delta - \beta\gamma -1) \otimes_K K[\mathcal{U}_0] \\
a & \longmapsto & \alpha\delta\otimes a + \beta\delta\otimes c -
\alpha\gamma\otimes b - \beta\gamma\otimes d \\
b & \longmapsto & \alpha^2\otimes b + \alpha\beta\otimes d -
\alpha\beta\otimes a - \beta^2\otimes c \\
c & \longmapsto & \gamma\delta\otimes a + \delta^2\otimes c -
\gamma^2\otimes b - \gamma\delta\otimes d \\
d & \longmapsto & \alpha\gamma\otimes b + \alpha\delta\otimes d -
\beta\gamma\otimes a - \beta\delta\otimes c
\end{eqnarray*}
Since the identity is a fixed point, this action restricts to one on $\mathcal{U}_{0,K}$. The projection map
$p\colon \SL{2}{K}\times\mathcal{U}_{0,K}\rightarrow\mathcal{U}_{0,K}$ is defined on coordinate rings
as
\begin{eqnarray*}
p\colon K[\mathcal{U}_0] & \longrightarrow & K[\alpha ,\beta ,\gamma ,
\delta]/(\alpha\delta - \beta\gamma -1) \otimes_K K[\mathcal{U}_0] \\
a &\longmapsto & 1\otimes a \\
b & \longmapsto & 1\otimes b \\
c & \longmapsto & 1\otimes c \\
d & \longmapsto & 1\otimes d.
\end{eqnarray*}
\end{definition}

To prove equivariance, it will be necessary to relate the conjugation
action to the following action of $\SL{2}{K}$ on $\mathbb{A}^2_K$.

\begin{definition}\label{def:affineaction}
The action 
\begin{align*}
\lambda\colon\SL{2}{K}\times\mathbb{A}^2_K &\longrightarrow
\mathbb{A}^2_K \\
\left( \begin{pmatrix} \alpha & \beta \\ \gamma &
    \delta \end{pmatrix}, (x,y)\right) & \longmapsto (\alpha x+\beta
y,\gamma x + \delta y)
\end{align*}
 is given on coordinate rings by
\begin{eqnarray*}
\lambda\colon K[x,y] & \longrightarrow & K[\alpha ,\beta ,\gamma ,\delta
]/(\alpha\delta - \beta\gamma -1)\otimes_K K[x,y] \\
x & \longmapsto & \alpha\otimes x + \beta\otimes y \\
y & \longmapsto & \gamma\otimes x + \delta\otimes y.
\end{eqnarray*}
\end{definition}

\begin{lemma}\label{lemma:coveringequivariance}
Both of the \'etale covers $f$ and $f^\prime$ defined in
Section~\ref{section:mainresult} (see Lemmas~\ref{lemma:defnf}
and~\ref{lemma:defnfprime}) are morphisms of $G$-spaces wth respect to
the conjugation action of Definition~\ref{def:actions} and the linear
action of Definition~\ref{def:affineaction} on $\mathcal{U}_{0,K}$ and
$\mathbb{A}^2_K$, respectively. That is, the diagram
\[
\xymatrix{
\SL{2}{K}\times\mathbb{A}^2_K \ar[r]^{\ \ \ \  \lambda} \ar[d]_{1\times
  f} & \mathbb{A}^2_K \ar[d]^{f} \\
\SL{2}{K}\times\mathcal{U}_{0,K} \ar[r]_{\ \ \ \  c} & \mathcal{U}_{0,K},}
\]
as well as the same diagram with $1\times f^\prime$ along the left
side and $f^\prime$ along the right, commutes.
\end{lemma}
\begin{proof}
By brute calculation.
\end{proof}

\begin{proposition}\label{prop:equivariance}
The local systems $\oE$ and $\oE^\prime$ of
Definition~\ref{def:thelocalsystems} are $\SL{2}{K}$-equivariant
sheaves on $\mathcal{U}_{0,K}$, with respect to the conjugation action.
\end{proposition}
\begin{proof}
We give the proof for $\oE$ only; the one for $\oE^\prime$ is totally analogous. Once again, the main tool for the proof is the equivalence between
local systems and \'etale fundamental group
representations. Therefore, we begin by fixing a geometric point
$\cl{x}\colon\Spec{\cl{K}}\rightarrow \SL{2}{K}\times \mathcal{U}_{0,K}$ of
the product:
\begin{eqnarray*}
\cl{x}\colon K[\alpha ,\beta ,\gamma ,
\delta]/(\alpha\delta - \beta\gamma -1) &\longrightarrow & \cl{K} \\
(\alpha ,\beta ,\gamma ,\delta ,a,b,c,d) &\longmapsto &
(1,0,1,0,1,1,0,1).
\end{eqnarray*}
Note that $c(\cl{x}) = \cl{u}$ and $p(\cl{x}) = \cl{u}$, where
$\cl{u}$ is the ($K$-rational) geometric point of $\mathcal{U}_{0,K}$ fixed
throughout this chapter. Thus, we have two functorial maps between
\'etale fundamental groups
\[
\algfg{\SL{2}{K}\times\mathcal{U}_{0,K},\cl{x}}
\stackrel{\operatorname{\pi_1}(\Conj )}{\longrightarrow}
\algfg{\mathcal{U}_{0,K},\cl{u}}
\]
and 
\[
\algfg{\SL{2}{K}\times\mathcal{U}_{0,K},\cl{x}}
\stackrel{\operatorname{\pi_1}(p)}{\longrightarrow}
\algfg{\mathcal{U}_{0,K},\cl{u}}.
\]
The inverse images $\Conj^*\oE$ and $p^*\oE$ are local systems, equivalent
to the characters $\rho_{\oE}\circ\operatorname{\pi_1}(\Conj )$ and
$\rho_{\oE}\circ\operatorname{\pi_1}(p)$ of
$\algfg{\SL{2}{K}\times\mathcal{U}_{0,K},\cl{x}}$, respectively. From the
definition of the functorial map and the alternate characterizations
of $\rho_{\oE}$, it follows that these characters are
defined entirely through characters of the automorphism groups of
covers of $\SL{2}{K}\times\mathcal{U}_{0,K}$ we shall denote $C$ and
$P$. They are obtained by pullback from the cover used to define
$\oE$, like so:
\[
\xymatrix{
C = \SL{2}{K}\times\mathcal{U}_{0,K}\times_{\mathcal{U}_{0,K}}\mathbb{A}^2_K
\ar[r] \ar[d] & \mathbb{A}^2_K \ar[d]^{f} \\
\SL{2}{K}\times\mathcal{U}_{0,K} \ar[r]_{c} & \mathcal{U}_{0,K}}
\]
\[
\xymatrix{
P = \SL{2}{K}\times\mathcal{U}_{0,K}\times_{\mathcal{U}_{0,K}}\mathbb{A}^2_K
\ar[r] \ar[d] & \mathbb{A}^2_K \ar[d]^{f} \\
\SL{2}{K}\times\mathcal{U}_{0,K} \ar[r]_{p} & \mathcal{U}_{0,K}}.
\]
The isomorphism
\[
\Conj^*\oE \iso p^*\oE
\]
exists because the covers $C$ and $P$ are isomorphic, and thus have
isomorphic automorphism groups. The map $C\rightarrow P$ is defined on
coordinates by
\begin{eqnarray*}
\alpha & \longrightarrow & \alpha \\
\beta & \longrightarrow & \beta \\
\gamma & \longrightarrow & \gamma \\
\delta & \longrightarrow & \delta \\
a & \longrightarrow & a \\
b & \longrightarrow & b \\
c & \longrightarrow & c \\
d & \longrightarrow & d \\
x & \longrightarrow & \delta\otimes x - \beta\otimes y \\
y & \longrightarrow & \alpha\otimes y - \gamma\otimes x.
\end{eqnarray*}
This does indeed define a homomorphism since, in the coordinate ring
for $C$, the first eight coordinates are related to the images of $x$
and $y$ under $f$ by the conjugation map, so by
Lemma~\ref{lemma:coveringequivariance}, the images of $x$ and $y$ must
be their images under the inverse group action on
$\mathbb{A}^2_K$. The inverse map $P\rightarrow C$ is defined as the
identity on all coordinates except $x$ and $y$, where
\begin{eqnarray*}
x & \longrightarrow & \alpha\otimes x + \beta\otimes y \\
y & \longrightarrow & \gamma\otimes x + \delta\otimes y.
\end{eqnarray*}
Then the composition of these maps is
\begin{eqnarray*}
x & \longrightarrow & \delta\alpha\otimes x - \beta\gamma\otimes x +
\delta\beta\otimes y - \beta\delta\otimes y \\
& & = (\alpha\delta - \beta\gamma )\otimes x \\
& & = 1\otimes x, \\
y & \longrightarrow & \alpha\delta\otimes x +\alpha\delta\otimes y -
\alpha\delta\otimes x - \gamma\beta\otimes y \\
& & = (\alpha\delta - \gamma\beta )\otimes y \\
& & = 1\otimes y
\end{eqnarray*}
and therefore the two maps are inverse isomorphisms, proving the
claim. Thus, the two composite characters are the same, proving there
is an isomorphism
\[
\Conj^*\oE \iso p^*\oE
\]
as desired.
\end{proof}

   \chapter{The Main Result}\label{ch:mainresult}

Finally we arrive at the main result of the thesis. It concerns the
nearby cycles of the local systems $\oE$ and $\oE^\prime$ defined in
the previous chapter taken with respect to two different integral
models for the regular unipotent subvariety $\mathcal{U}_{0,K}$.
But first, we describe the nearby cycles functor itself.

Throughout Chapter~\ref{ch:mainresult}, we retain the notation
established in the previous chapter. In particular, $K$ is a $p$-adic
field of mixed characteristic and $K^\prime$ is the quadratic
extension of $K$ obtained by adjoining a square root of a uniformizer
$\varpi$ of $K$'s ring of integers, $\mathcal{O}_K$.

\section{The Nearby Cycles Functor}\label{section:nearbytransport}

Chapter~\ref{ch:somelocalsystems} introduced a pair of local systems on the
regular unipotent subvariety of $\SL{2}{K}$ ($K$ a local field of
mixed characteristic) to which we will associate
distributions in Chapter~\ref{ch:distributions}. The first step in
that process makes use of the Grothendieck-Lefschetz trace formula,
which only makes sense for local systems over the residue field $k$ of
$K$. Therefore, what's required is a mechanism that produces local
systems on $\SL{2}{k}$ from local systems on $\SL{2}{K}$. The nearby
cycles functor is just such a mechanism, one especially well-suited to
this purpose because it also allows the local system on $\SL{2}{k}$ to
retain the action of $\Gal{\cl{K}/K}$ present on the original local
system that was studied in Chapter~\ref{ch:galoisactions}.

SGA 7(II)~\cite{SGA7ii} Expos\'e XIII is the fundamental source for
the material in this section. There, Deligne defines the nearby cycles functor for
torsion sheaves on a scheme defined over a Henselian trait.

Let $S$ be a Henselian trait (\ie, the spectrum of a Henselian
discrete valuation ring $R$) with closed point $s$ and generic point
$\eta$, so $s$ is the spectrum of the residue field of $R$ and $\eta$ is the spectrum of
the quotient field $K$ of $R$. Choose a separable closure $\cl{K}$ of $K$ to
define the geometric generic fibre $\cl{\eta}$. Let
$\cl{S}$ be the spectrum of the valuation ring of $\cl{K}$ and
$\cl{s}$ the residue field of that ring. Let $\ell$ be relatively prime to the
residue characteristic of $R$, $p$. Let $X$ be a scheme over $S$ with
generic and special fibres $X_\eta$ and $X_s$ as in the following diagram (both
squares are Cartesian).
\[
\xymatrix{
X_\eta \ar@{^{(}->}[r]^j \ar[d] & X \ar[d] & X_s \ar@{_{(}->}[l]_i \ar[d] \\
\eta \ar@{^{(}->}[r] & S & s \ar@{_{(}->}[l] }
\]
Finally, let $\cl{X}$, $X_\cl{\eta}$, and $X_\cl{s}$ be the respective
base extensions of $X$, $X_\eta$, and $X_s$ to $\cl{S}$, $\cl{\eta}$,
and $\cl{s}$ as defined by the Cartesian diagram below, which sits
over the diagram above.
\[
\xymatrix{
X_{\cl{\eta}} \ar@{^{(}->}[r]^{\cl{j}} \ar[d] & \cl{X} \ar[d] & X_{\cl{s}} \ar@{_{(}->}[l]_{\cl{i}} \ar[d] \\
\cl{\eta} \ar@{^{(}->}[r] & \cl{S} & \cl{s} \ar@{_{(}->}[l] }
\]
In SGA, the nearby cycles functor is defined for torsion sheaves $\mathcal{F}$ on
$X$, and consists of two component sheaves, one taking the inverse
image of $\mathcal{F}$ on $X_s$ as input and the other taking the
inverse image of $\mathcal{F}$ on $X_\eta$. This latter
component is the one that we will define and use as the nearby cycles
functor. It carries torsion sheaves on $X_\eta$ into a topos denoted
$X_{\cl{s}}\times_s\eta$.
\begin{definition}\label{def:wackytopos}
An element of the topos $X_s\times_s\eta$ is a sheaf
$\mathcal{F}_\cl{\eta}$ on $X_s$, the subscript indicating that it is
  equipped with a continuous action of $\Gal{\cl{\eta}/\eta}$
  compatible with the action on $X_\cl{s}$ (via the continuous map
  $\Gal{\cl{\eta}/\eta}\rightarrow \Gal{\cl{s}/s}$).
\end{definition} 
\begin{definition}\label{def:Psi}
Let $\mathcal{F}$ be a torsion sheaf on $X_\eta$ and
$\mu_{\mathcal{F}}$ the natural continuous, compatible action (in the sense of
Definitions~\ref{def:compatibleaction} and~\ref{def:ctsaction}) on the
inverse image $\cl{\mathcal{F}}$ of $\mathcal{F}$ on $X_{\cl{\eta}}$
(as explained in the paragraphs following
Definition~\ref{def:ctsaction}). The \emph{nearby cycles functor}
$\Psi$ is defined by
\[
\mathcal{F}\longmapsto \Psi (\mathcal{F})\ceq
(\cl{i}^*\cl{j}_*\cl{\mathcal{F}} , \Psi (\mu_{\mathcal{F}})),
\]
where $R\Psi (\mu_{\mathcal{F}})$ is the continuous, compatible action
of $\Gal{\cl{\eta}/\eta}$ on $\cl{i}^*\cl{j}_*\cl{\mathcal{F}}$
obtained from $\mu_{\mathcal{F}}$ by `transport of structure,' a vahue
phrase deserving some illumination.
\end{definition}

To wit, for each $\gamma\in\Gal{\cl{\eta}/\eta}$ with image
$\cl{\gamma}$ under the canonical map
$\Gal{\cl{\eta}/\eta}\rightarrow\Gal{\cl{s}/s}$, we will specify an isomorphism
\[
\operatorname{\Psi}(\mu_{\mathcal{F}} )(\gamma )\colon\cl{\gamma}_*\cl{i}^*\cl{j}_*
\mathcal{F_{\cl{\eta}}}\longrightarrow \cl{i}^*\cl{j}_* \mathcal{F}_{\cl{\eta}}
\]
derived from the given isomorphism
\[
\mu_{\mathcal{F}}(\gamma )\colon \gamma_*\mathcal{F}_{\cl{\eta}}\longrightarrow \mathcal{F}_{\cl{\eta}}.
\]
This notation is somewhat deceptive because the isomorphism is not
just the image of $\mu_{\mathcal{F}}$ under the functor $\Psi$. That is, however, where things start:
\[
\Psi (\mu (\gamma ))\colon \cl{i}^*\cl{j}_*\gamma_*\mathcal{F}_{\cl{\eta}}\longrightarrow\cl{i}^*\cl{j}_*\mathcal{F}_{\cl{\eta}}.
\]
Then, we use the fact that
\[
\cl{j}_*\gamma_* = (\cl{j}\gamma )_* = (\gamma\cl{j})_* = \gamma_*\cl{j}_*
\]
because we have the commuting square
\[
\xymatrix{
\cl{\eta} \ar[r]^{\cl{j}} \ar[d]_{\gamma} & \cl{S} \ar[d]^{\gamma} \\
\cl{\eta} \ar[r]_{\cl{j}} & \cl{S} }
\]
where the $\gamma$ along the right is the scheme morphism induced by
the restriction of the Galois morphism $\gamma$ to the ring $R$. So we
can think of $R\Psi (\mu (\gamma ))$ as a morphism
\[
\cl{i}^*\gamma_*\cl{j}_*\mathcal{F}_{\cl{\eta}} \longrightarrow \cl{i}^*\cl{j}_*\mathcal{F}_{\cl{\eta}}.
\]
The final step is provided by the proper base change theorem
(Proposition~\ref{thm:pbc}), using the commuting square
\[
\xymatrix{
\cl{S} \ar[d]_{\gamma} & \cl{s} \ar[l]_{\cl{i}} \ar[d]^{\cl{\gamma}} \\
\cl{S} & \cl{s} \ar[l]^{\cl{i}} }
\]
where $\cl{\gamma}$ is the scheme morphism induced by the image of
$\gamma$ under the canonical homomorphism $\Gal{\cl{\eta}/\eta}
\rightarrow \Gal{\cl{s}/s}$. The transported Galois isomorphism is the composite
\[
\cl{\gamma}_*\cl{i}^*\cl{j}_*\mathcal{F}_{\cl{\eta}} \iso \cl{i}^*\gamma_*\cl{j}_*\mathcal{F}_{\cl{\eta}} \longrightarrow \cl{i}^*\cl{j}_*\mathcal{F}_{\cl{\eta}},
\]
the leftmost map coming from proper base change (see Theorem~\ref{thm:pbc}).

The functor $\Psi$ has an obvious extension to the category of
$\pi$-adic sheaves, and from there to $\cl{\QQ}_\ell$-local systems,
using the definition of continuous, compatible actions extended to
these categories in Chapter~\ref{ch:galoisactions}
Section~\ref{section:systemactions}. 
\begin{remark}\label{rem:integralmodeldependence}
As mentioned above, the original framework set by Deligne specifies the
scheme $X$ over the Henselian trait $S$ at the outset, but the
definition of $\Psi$ given above only makes reference to $X_\eta$ and
$X_s$. Indeed, the work in this chapter begins with only the generic
fibre $X_\eta$, leaving open the possibility of defining nearby cycles
functors with respect to different integral models of $X_\eta$, \ie,
any scheme $\un{X}$ over $S$ equipped with an isomorphism
$\un{X}_\eta\iso X_\eta$. Theorem~\ref{prop:mainlocalsystems} details
what happens when we take the nearby cycles of the local systems $\oE$
and $\oE^\prime$ from Chapter~\ref{ch:somelocalsystems} with respect
to two different integral models of $\mathcal{U}_{0,K}$. These integral models
are defined in the next section.
\end{remark}

\section{Integral Models}\label{section:integralmodels}


Deligne defined nearby cycles for torsion sheaves on a scheme defined
over a Henselian trait. Therefore, in the diagram
\[
\xymatrix{
X_{{\eta}} \ar@{^{(}->}[r]^{j} \ar[d] & X \ar[d] & X_s
\ar@{_{(}->}[l]_{i} \ar[d] \\
\eta \ar@{^{(}->}[r] & S & s \ar@{_{(}->}[l] }
\] 
the middle scheme $X$ over the trait comes specified at the
outset. 

This thesis uses a `restriction' of the nearby cycles
functor, and begins only with a scheme over the field $K$, that is, with
the generic fibre of the diagram. Thus, we are free to choose any
integral model for the starting scheme and calculate nearby cycles
with respect to that model. In the next section we will use two different models.
\begin{definition}
The \emph{standard integral model} $\un{\mathcal{U}}_0$ of
$\mathcal{U}_{0,K}$  is obtained from the affine scheme
\[
\un{\mathcal{U}}\ceq \mathcal{U}_{0,\mathcal{O}_K} = \Spec{\mathcal{O}_K[a,b,c,d]/(ad-bc-1,a+d-2)}
\]
by removing the identity. It can be built with a gluing construction
analogous to that given for $\mathcal{U}_{0,K}$ (in \S~\ref{ssection:schemes}), with $K$ replaced
by $\mathcal{O}_K$ throughout.\par
The \emph{nonstandard integral model} $\un{\mathcal{U}}^\prime_0$ is obtained
from the affine scheme 
\[
\un{\mathcal{U}}^\prime \ceq
\Spec{\mathcal{O}_K\left[a,\varpi b,\frac{c}{\varpi},d\right]/(ad-bc-1,a+d-2)}
\]
by likewise removing the identity. It too can be constructed by a
gluing analogous to $\un{\mathcal{U}}_0$ and $\mathcal{U}_0$.
\end{definition}
In accord with Remark~\ref{rem:affineforquasiaffine}, we will work
with these models as if they were affine, conflating their coordinate
rings with those of the affine schemes $\un{\mathcal{U}}$ and
$\un{\mathcal{U}}^\prime$, respectively. For compactness, we will
denote these rings by $\mathcal{O}_K[\un{\mathcal{U}}_0]$ and
$\mathcal{O}_K[\un{\mathcal{U}}_0^\prime ]$.

Here is how these models play into the main result: Recall that for a
local system $\mathcal{F}$ on $\mathcal{U}_{0,K}$, $\Psi\mathcal{F} =
(\cl{i}^*\cl{j}_*\cl{\mathcal{F}}, \Psi (\mu ))$, where
$\cl{\mathcal{F}}$ is the inverse image of $\mathcal{F}$ under the
canonical projection $\cl{\mathcal{U}}_0 \rightarrow \mathcal{U}_{0,K}$,
$\mu$ is the natural Galois action on $\cl{\mathcal{F}}$ induced by base
change, and $\cl{i}$ and $\cl{j}$ are the
inclusions of the special and generic fibres of a chosen integral
model for $\mathcal{U}_{0,K}$. Since it is the inverse image of a local
system, $\cl{\mathcal{F}}$ is a local system, and so, in fact, is
$\cl{i}^*\cl{j}_*\cl{\mathcal{F}}$, as will be shown in the proof of
Theorem~\ref{prop:mainlocalsystems}. If the character
$\Gal{\cl{K}/K}$ that defines the transported Galois action
$\Psi (\mu )$ factors through the canonical map
$\Gal{\cl{K}/K}\rightarrow \Gal{\cl{k}/k}$, giving a continuous action
of $\Gal{\cl{k}/k}$, then the extension of Deligne's equivalence proved in
Chapter~\ref{ch:galoisactions} shows that the nearby cycles is
equivalent to a local system on $\mathcal{U}_{0,k}$. In this case, we
say the nearby cycles \emph{descends}.
\begin{definition}\label{def:adapted}
When the nearby cycles of a local system, computed with respect to a particular integral
model for $\mathcal{U}_{0,K}$, descends to a local system on $(\mathcal{U}_0)_k$ we will say
that the integral model is \emph{adapted} to the local system. 
\end{definition}

\section{The Main Result}\label{section:mainresult}
\begin{theorem}\label{prop:mainlocalsystems}
The integral model $\un{\mathcal{U}}_0$ is adapted to the local system
$\oE$ but not to $\oE^\prime$, while $\un{\mathcal{U}}^\prime_0$ is
adapted to $\oE^\prime$ but not to $\oE$.
\end{theorem}
\begin{proof}
The strategy of the proof comes from the following observation,
justified at the end of the proof: that if the nearby cycles of a
local system $\mathcal{L}$ descends, the local system on
$\mathcal{U}_{0,k}$ it descends to is $i^*j_*\mathcal{L}$, implying
that this image sheaf is itself a local system. This is not the case generally, the obstruction being
the finicky pushforward operation. 

While inverse images of local systems are
always local systems, and can be calculated at the level of
fundamental group representations simply by composing the
representation of the original sheaf with the functorial map between
fundamental groups, the pushforward of a local system need not be a
local system. Let $f\colon X\rightarrow Y$ be a morphism of schemes
and $\mathcal{F}$ a local system on $X$. For $f_*\mathcal{F}$ to be a
local system, the representation of $\algfg{X,\cl{x}}$ equivalent to
$\mathcal{F}$ must factor through the functorial map
$\pi_1(f)\colon\algfg{X,\cl{x}}\rightarrow \algfg{Y,\cl{y}}$, as in
\[
\xymatrix{\algfg{X,\cl{x}} \ar[dr]_{\rho_\mathcal{F}}
  \ar[rr]^{\pi_1(f)} && \algfg{Y,\cl{y}}
  \ar[dl]^{\rho_{f_*\mathcal{F}}} \\
& \cl{\QQ}^\times_\ell , & }
\]
in which case the factorizing character of $\algfg{Y,\cl{y}}$ is equivalent to
$f_*\mathcal{F}$. Fortunately, the covers used to define $\oE$ and $\oE^\prime$
arise from covers of $\underline{\mathcal{U}}_0$ and
$\un{\mathcal{U}}_0^\prime$ via base
change, which allows us to easily determine when the defining
characters factor in this way and when they don't, from which the
conclusions of the theorem derive.

In particular, the covers permitting these
factorizations are by the affine plane over $\mathcal{O}_K$
defined completely analogously to $f$ and $f^\prime$ (from
Lemmas~\ref{lemma:defnf} and~\ref{lemma:defnfprime}) for the
standard and nonstandard integral models, respectively. On
(falsely affine) coordinate rings,
\begin{eqnarray*}
\un{f}\colon \mathcal{O}_K[\un{\mathcal{U}}] & \longrightarrow &
\mathcal{O}_K[x,y] \\
a &\longmapsto & 1 -xy \\
b & \longmapsto & x^2 \\
c & \longmapsto & -y^2 \\
d & \longmapsto & 1 + xy
\end{eqnarray*}
\begin{eqnarray*}
\un{f}^\prime\colon \mathcal{O}_K[\un{\mathcal{U}}^\prime ] & \longrightarrow &
\mathcal{O}_K[x,y] \\
a & \longmapsto & 1 - xy \\
b & \longmapsto & \frac{x^2}{\varpi} \\
c & \longmapsto & -\varpi y^2 \\
d & \longmapsto & 1 + xy.
\end{eqnarray*}
Both of these covers have automorphism groups of order 2, with maps
defined exactly as those for $f$ and $f^\prime$, which we will denote
by $\un{\id}$ and $\un{(-1)}$. The Cartesian square determined by the inclusion of the generic fibre into
the standard integral model and the cover $\un{f}$ is
\[
\xymatrix{
K[x,y] & \mathcal{O}_K[x,y] \ar[l]_{\delta} \\
\mathcal{O}_K[\mathcal{U}] \ar[u]^f & \mathcal{O}_K[\un{\mathcal{U}}] \ar[l] \ar[u]_{\un{f}} }
\]
and the one determined by the inclusion into the nonstandard model
and $\un{f}^\prime$ is
\[
\xymatrix{
K[x,y] & \mathcal{O}_K[x,y] \ar[l]_\delta \\
\mathcal{O}_K[\mathcal{U}] \ar[u]^{f^\prime} & \mathcal{O}_K[\un{\mathcal{U}}^\prime ] \ar[l] \ar[u]_{\un{f}^\prime}.
}
\]
In both cases, the homomorphism $\delta$ is simply the natural
inclusion map. Therefore, the induced maps of automorphism groups are
\begin{eqnarray*}
\Aut_{\mathcal{U}_{0,K}}(f) & \longrightarrow & \Aut_{\un{\mathcal{U}}_0}(\un{f}) \\
\id &\longmapsto & \un{\id} \\
(-1) & \longmapsto & \un{(-1)}.
\end{eqnarray*}
and
\begin{eqnarray*}
\Aut_{\mathcal{U}_{0,K}}(f^\prime ) &\longrightarrow & \Aut_{\un{\mathcal{U}}^\prime_0}(\un{f}^\prime ) \\
\id & \longmapsto & \un{\id} \\
(-1) & \longmapsto & \un{(-1)},
\end{eqnarray*}
so that the nontrivial characters $\chi$ and $\chi^\prime$ of $\Aut_{\mathcal{U}_{0,K}}(f)$ and $\Aut_{\mathcal{U}_{0,K}}(f^\prime )$
factor through the nontrivial characters $\un{\chi}$ and $\un{\chi}^\prime$ of $\Aut_{\un{\mathcal{U}}_0}(\un{f})$ and
$\Aut_{\un{\mathcal{U}}^\prime_0}(\un{f}^\prime )$, respectively. If we define $\un{\oE}$ and
$\un{\oE}^\prime$ to be the local systems on $\un{\mathcal{U}}_0$ and
$\un{\mathcal{U}}_0^\prime$ defined by these characters, we can
conclude that both $(j_{\un{\mathcal{U}}_0})_*\oE$ and
$(j_{\un{\mathcal{U}}^\prime_0} )_*\oE^\prime$ are local systems and
specifically that
\[
(j_{\un{\mathcal{U}}_0})_*\oE = \un{\oE}
\]
and
\[
(j_{\un{\mathcal{U}}^\prime_0} )_*\oE^\prime = \un{\oE}^\prime.
\]
From here, there are no obstructions to the descent of the nearby cycles,
since the inverse image of a local system always gives a local system.
Composing $\un{\chi}$ with the functorial map
$\pi_1(i_{\un{\mathcal{U}}_0})\colon
\algfg{(\un{\mathcal{U}}_0)_s,\cl{u}} \rightarrow
\algfg{\cl{\mathcal{U}}_0,\cl{u}}$ gives a character equivalent to a
  local system on $(\un{\mathcal{U}}_0)_s = \mathcal{U}_{0,k}$ which we denote
  $\oE_s$. Let $\cl{\oE_s}$ be the base change of $\oE_s$ to
  $\cl{(\un{\mathcal{U}}_0)_s} \ceq
  (\un{\mathcal{U}}_0)_s\times_k\cl{k}$. Then, 
\[
\Psi_{\un{\mathcal{U}}_0}\oE = (\cl{\oE_s}, \mu ),
\]
where $\mu$ is the natural action of
$\Gal{\cl{K}/K}$ on $\cl{\oE_s}$ arising from base change. And composing $\un{\chi}^\prime$
with the functorial map
$\pi_1(i_{\un{\mathcal{U}}_0^\prime})\colon\algfg{(\un{\mathcal{U}}^\prime_0)_s,\cl{u}}\rightarrow
\algfg{\un{\mathcal{U}}^\prime_0,\cl{u}}$ gives a character equivalent to a
local system on $(\un{\mathcal{U}}^\prime_0)_s$, denoted
$\oE^\prime_s$. From this we conclude that 
\[
\Psi_{\un{\mathcal{U}}^\prime_0}\oE^\prime = (\cl{\oE^\prime_s}, \mu^\prime).
\]
Under the equivalence of Proposition~\ref{prop:delignegaloissheaf}
extended to local systems in Chapter~\ref{ch:galoisactions}, both of
these nearby cycles local systems descend to local systems on
$(\un{\mathcal{U}}_0)_s = (\un{\mathcal{U}}^\prime_0)_s$, which must
be those local systems from which they were defined. This completes
the proof of one half of the theorem.

The second half, regarding the nearby cycles of $\oE$ and $\oE^\prime$
calculated with respect to the opposite integral model used
previously, is proved by returning to the question of factoring the
defining characters $\chi$ and $\chi^\prime$, this time factoring
$\chi$ through $\algfg{\un{\mathcal{U}}^\prime_0,\cl{u}}$ and
$\chi^\prime$ through $\algfg{\un{\mathcal{U}}_0,\cl{u}}$. In these
cases, the claim is that no such factorizations exist, accounting for
the failure of the nearby cycles to descend to local systems on
$(\un{\mathcal{U}}_0)_s = (\un{\mathcal{U}}^\prime_0)_s$.

The essential problem is that the only covers of the integral models
from which the covers used to define $\oE$ and $\oE^\prime$ can be
obtained via base change are not \'etale, and therefore unusable for
the purposes of factorizing $\chi$ and $\chi^\prime$. To see this
we begin by noting that we could have approached the successful
factorization above using the alternate covers $f_\varpi$ and
$f^\prime_\varpi$ through which $\oE$ and $\oE^\prime$ can also be
defined (by Lemma~\ref{lemma:varpicovers}). Both of those covers arise
by base change from covers of
$\un{\mathcal{U}}_0$ and $\un{\mathcal{U}}_0^\prime$, namely covers by
  $\mathbb{A}^2_{0,\mathcal{O}_{K^\prime}}$ with covering
maps defined in the same way as $\un{f}$ and $\un{f}^\prime$. However,
neither of these covers is \'etale (both ramify, for example, at the
points in $\un{\mathcal{U}}_0$ and $\un{\mathcal{U}}_0^\prime$
corresponding to the prime ideal $(p)$), so the question of
factorization could not be approached through those covers. But in
both cases there existed intermediate covers that did offer a means of
factorizing the characters because those intermediate covers over
$\mathcal{U}_{0,K}$ could also be used to define $\oE$ and $\oE^\prime$
and the corresponding covers over the integral models were, in fact,
\'etale. Similarly, there are Cartesian squares relating covers
through which our local systems can be defined and covers of the
standard and nonstandard integral models, namely
\[
\xymatrix{
\mathbb{A}^2_{0,K^\prime} \ar[d]_{f_\varpi} \ar[r] &
\mathbb{A}^2_{0,\mathcal{O}_{K^\prime}}
\ar[d]^{\un{f}^\prime_\varpi} \\
\mathcal{U}_{0,K} \ar[r] & \un{\mathcal{U}}_0^\prime }
\]
and
\[
\xymatrix{
\mathbb{A}^2_{0,K^\prime} \ar[d]_{f^\prime_\varpi} \ar[r] &
\mathbb{A}^2_{0,\mathcal{O}_{K^\prime}}
\ar[d]^{\un{f}_\varpi} \\
\mathcal{U}_{0,K} \ar[r] & \un{\mathcal{U}}_0. }
\]
But once again, this is not a usable setting to consider the question of
factorization since the covers of the integral models are not
\'etale. However, in this case there do not exist intermediate covers
that define Cartesian subsquares of the ones above. The possible
intermediate covers over $\mathcal{U}_{0,K}$ are $f$ for the first diagram
and $f^\prime$ for the second (because those are the only remaining
covers that can be used to define the local systems we want), and the
only possible covers over the integral models from which we could
recover a Cartesian square are $\un{f}^\prime$ for the first diagram
and $\un{f}$ for the second. However, the only possible map between
these pairs of covers that would complete such a square is defined
only over the quadratic extension $K^\prime /K$, not over $K$
itself and therefore cannot be used (in particular, the necessary map
would carry $x$ to $\frac{x}{\sqrt{\varpi}}$ and $y$ to
$\sqrt{\varpi}y$). Hence, the character $\chi$ cannot factor through
$\algfg{\un{\mathcal{U}}_0^\prime,\cl{u}}$ nor can $\chi^\prime$
factor through $\algfg{\un{\mathcal{U}}_0,\cl{u}}$. We can infer,
then, that the pushforward of $\oE$ to $\un{\mathcal{U}}_0^\prime$ is
not a local system, nor is the pushforward of $\oE^\prime$ to
$\un{\mathcal{U}}_0$ a local system.

This allows us to conclude that $\Psi_{\un{\mathcal{U}}_0^\prime}\oE$
and $\Psi_{\un{\mathcal{U}}_0}\oE^\prime$ do not descend to local
systems on $(\un{\mathcal{U}}_0)_s$. Justification for this conclusion
comes from the following two facts: first, that the sheaf components of
$\Psi_{\un{\mathcal{U}}_0^\prime}\oE$ and
$\Psi_{\un{\mathcal{U}}_0}\oE^\prime$ are both local systems on $\mathcal{U}_{0,\cl{k}}$; and
second, that if these local systems descended, the pushforwards that
we attempted to calculate above would be local systems. The first fact
is proved by first noting that the base changes of $\oE$ and
$\oE^\prime$ are local systems on
$\cl{\mathcal{U}}_0$, defined by the
compositions of $\chi$ and $\chi^\prime$ with the functorial map
\[
\algfg{\cl{\mathcal{U}}_0,\cl{u}} \longrightarrow
\algfg{\mathcal{U}_{0,K},\cl{u}}.
\]
From the calculations above, it can be seen that both of these
compositions are the same, equal to the nontrivial character of
$\algfg{\cl{\mathcal{U}}_0,\cl{u}}$. The first fact is proved after it
is shown that the pushforward of this local system to the integral
model 
\[
\cl{\un{\mathcal{U}}}_0 =
\Spec{\mathcal{O}_{K^\textrm{nr}}[a,b,c,d]/(ad-bc-1,a+d-2)},
\]
where $K^\textrm{nr}$ is the maximal unramified extension of $K$, is a
local system. To do this, we construct a factorization of the
nontrivial character of $\algfg{\cl{\mathcal{U}}_0,\cl{u}}$ through
the auromorphism group of the following \'etale cover of
$\cl{\un{\mathcal{U}}}_0$. Let $\mathbb{A}^2_{\mathcal{O}_{\cl{K}}}
\rightarrow \cl{\un{\mathcal{U}}}_0$ be given on coordinate rings by
\begin{eqnarray*}
\mathcal{O}_{\cl{K}}[a,b,c,d]/(ad-bc-1,a+d-2) & \longrightarrow &
\mathcal{O}_{\cl{K}}[x,y] \\
a & \longmapsto & 1 - xy \\
b & \longmapsto & x^2 \\
c & \longmapsto & -y^2 \\
d & \longmapsto & 1+xy.
\end{eqnarray*}
This is a degree 2 Galois cover with automorphism group consisting of
a nontrivial map that multiplies both coordinates by -1. The square
\[
\xymatrix{
\mathbb{A}^2_{\cl{K}} \ar[r] \ar[d] &
\mathbb{A}^2_{\mathcal{O}_{\cl{K}}} \ar[d] \\
\cl{\mathcal{U}}_0 \ar@{^{(}->}[r]^{\cl{j}} & \cl{\un{\mathcal{U}}}_0
}
\]
is Cartesian, and the top map defines a map of automorphism groups
that sends nontrivial map to nontrivial map, allowing the base changed
character to factor. Therefore, $\cl{j}_*\cl{\oE} =
\cl{j}_*\cl{\oE^\prime}$ is a local system, and hence so is its
inverse image on the special fibre.

If $\cl{j}_*\cl{\oE}$ descends to local system on
$(\un{\mathcal{U}}_0)_s$, its
pushforwards to the integral model the nearby cycles were taken with
respect to would be a local system, and equal to the pushforwards from the generic
fibre we have attempted to calculate. This is because the trait $S =
\Spec{\mathcal{O}_K}$ both integral models are defined over is
Henselian, and the \'etale site of a scheme defined over a Henselian trait
is equivalent to the \'etale site of its special fibre (see
Murre~\cite[8.1.3]{murre}). This result renders the factorization
requirement for pushforwards of local systems to be local systems
trivial for pushforwards from the special fibre (covers of the special
fibre \emph{always} arise via base change from the integral
model and each such pair of covers will have isomorphic covering
groups). This completes the proof.
\end{proof}

   \chapter{Distributions From The Local Systems}\label{ch:distributions}

As an application of Theorem~\ref{prop:mainlocalsystems}, the last
piece of business in this thesis is to associate
distributions to $\oE$ and $\oE^\prime$ and
establish some of the basic properties of those distributions. In
particular, it turns out that they are admissible in the sense of
Harish-Chandra~\cite{HC78} and are eigendistributions of
the Fourier transform.

As in the previous chapter, set $G=\SL{2}{K}$; in addition, we write
$G_k$ for $\SL{2}{k}$.

\section{Distributions from the Local Systems}
By Theorem~\ref{prop:mainlocalsystems}, the nearby cycles sheaves
\[
\mathcal{E} \ceq R\Psi_{\un{\mathcal{U}}_0}\oE\textrm{  and
}\mathcal{E}^\prime \ceq R\Psi_{\un{\mathcal{U}}^\prime_0}\oE^\prime
\]
on $\mathcal{U}_{0,\cl{k}}$ descend to local systems on
$\mathcal{U}_{0,k}$. Both $\mathcal{E}$ and $\mathcal{E}^\prime$
consist of two pieces of data, a local system on
$\mathcal{U}_{0,\cl{k}}$ and a continuous action of
$\Gal{\cl{k}/k}$. The proof of Theorem~\ref{prop:mainlocalsystems}
showed that both components of $\mathcal{E}$ and
$\mathcal{E}^\prime$ are equal. Despite this we will continue refer to
both because distince objects will eventually be defined from these
two nearby cycles sheaves. Denote the local system component of
both sheaves by $\cl{\mathcal{E}}$. That $\mathcal{E}$ and
$\mathcal{E}^\prime$ descend means there is an isomorphisms of sheaves
\[
\varphi\colon\Fr^*\cl{\mathcal{E}}\stackrel{\sim}{\longrightarrow}\cl{\mathcal{E}}
\]
and
\[
\varphi^\prime\colon\Fr^*\cl{\mathcal{E}}^\prime\stackrel{\sim}{\longrightarrow}\cl{\mathcal{E}}^\prime
,
\]
where $\Fr$ is the automorphism of $G_{\cl{k}}$ induced by the
Frobenius in $\Gal{\cl{k}/k}$. Then, as
described in SGA$4\frac{1}{2}$~\cite[\emph{Rappels}]{SGA4half}, it is possible to define a characteristic
function for each of $\mathcal{E}$ and $\mathcal{E}^\prime$.
\begin{definition}
The \emph{trace of Frobenius}
associated to a local system $\mathcal{F}$ on $\mathcal{U}_{0,\cl{k}}$
with an isomorphism
\[
\varphi_{\mathcal{F}}\colon\Fr^*\mathcal{F}\stackrel{\sim}{\longrightarrow}\mathcal{F}
\]
is given by 
\[
t_\Fr^{\mathcal{F}}(a) \ceq
\operatorname{Trace}((\varphi_{\mathcal{F}})_{\cl{a}};\mathcal{F}_{\cl{a}}), \qquad
\forall a\in \mathcal{U}_{0,\cl{k}}(k).
\]
To be precise, the value $t^{\mathcal{F}}_\Fr(a)$ is the trace of the
operator on $\mathcal{F}_{\cl{a}}$ induced by the composition of
morphisms
\begin{equation}\label{eqn:stalkoperator}
\mathcal{F}_{\cl{a}} = \mathcal{F}_{\Fr(\cl{a})} \stackrel{\sim}{\longrightarrow}
(\Fr^*\mathcal{F})_{\cl{a}} \stackrel{\varphi_\cl{a}}{\longrightarrow}
\mathcal{F}_{\cl{a}},
\end{equation}
the leftmost map being the canonical isomorphism and
$\varphi_{\cl{a}}$ the isomorphism of stalks induced by $\varphi$.
\end{definition}
Thus, both $t_\Fr^{\mathcal{E}}$ and $t_\Fr^{\mathcal{E}^\prime}$ are
functions on $\mathcal{U}_{0,k}(k)$. Extend these functions by zero
to $G_k(k)$. From there, we can obtain functions on $G(K)$ using the
parahoric subgroups of $\SL{2}{K}$ that sit above the integral models
$\un{\mathcal{U}}_0$ and $\un{\mathcal{U}}_0^\prime$. These are
\begin{eqnarray*}
 X &\ceq & \Spec{\mathcal{O}[a,b,c,d]/(ad-bc-1)} \\
 X^\prime & \ceq & \Spec{\mathcal{O}\left[ a,\varpi
   b,\frac{c}{\varpi},d\right] /(ad-bc-1)}, 
\end{eqnarray*}
respectively. Inflate $t_\Fr^{\mathcal{E}}$ to $X(\mathcal{O}_K)$ and
$t_\Fr^{\mathcal{E}^\prime}$ to $X^\prime(\mathcal{O}_K)$) by
composing with the reduction map
\[
X(\mathcal{O}_K)\textrm{
  (resp. $X^\prime(\mathcal{O}_K)$)}\longrightarrow G_k(k).
\]
Finally, extend the results by zero to obtain functions
\[
f_{\oE},f_{\oE^\prime}\colon G(K)\longrightarrow
\cl{\QQ}_\ell .
\]
Each of these functions is supported by a compact subgroup---either
$X(\mathcal{O}_K)$ or $X^\prime (\mathcal{O}_K)$---and the value of any
point in the support is determined by its value under the reduction
map. Therefore, both $f_{\oE}$ and $f_{\oE^\prime}$ belong to the
Hecke algebra of $G(K)$.
\begin{definition}
The \emph{Hecke algebra} $\mathcal{C}(G(K))$ of $G(K)$ is the set of compactly supported
and locally constant functions
\[
f\colon G(K)\longrightarrow \cl{\QQ}_\ell
\]
endowed with the structure of a $\cl{\QQ}_\ell$-algebra with scalar
multiplication, pointwise addition and convolution of functions.
\end{definition}
\begin{remark}\label{rem:parahoricdependence}
It is worth noting here that although the trace of Frobenius functions
$t^{\mathcal{E}}_\Fr$ and $t^{\mathcal{E}^\prime}_\Fr$ are identical,
they were inflated through different parahoric subgroups. The
resulting functions $f_{\oE}$ and $f_{\oE^\prime}$ are
distinct; therefore, the subscripts are no longer meaningless,
keeping track of the parahoric subgroup the trace was
inflated to. The switch to $\oE$ and $\oE^\prime$ is meant to
emphasize the connection to the original local systems on $G$ that
give rise to these functions.
\end{remark} 

Throughout the rest of this chapter we fix a Haar measure on $G(K)$.

\begin{definition}\label{def:dists}
Let $\mathcal{L}\in\{\oE , \oE^\prime\}$. Then for all $f\in\mathcal{C}(G(K))$,
\[
\Theta_{\mathcal{L}}(f) \ceq \int_{G(K)}\int_{G(K)}
f(y^{-1}xy)f_{\mathcal{L}}(x)dxdy.
\]
\end{definition}

\section{Fundamental Properties of the Distributions}

Now we prove that the integrals in Definition~\ref{def:dists} do
indeed converge, giving distributions, and that these distributions
are \emph{admissible} as defined by Harish-Chandra in~\cite[$\S
14$]{HC78}. There, he goes on to note that all distributions obtained
as characters of admissible representations (a process described in~\cite[$\S
5$]{HC70}) as well as all linear
combinations of such distributions are admissible. It turns
out that $\Theta_{\oE}$ and $\Theta_{\oE^\prime}$ are just such
objects. In light of that, recalling the exact and somewhat technical definition of
admissibility is something of an unnecessary distraction, so we will refrain. The
curious reader can consult Harish-Chandra's article.

\begin{proposition}\label{prop:properties1}
Let $\mathcal{L}\in\{\mathcal{E},\mathcal{E}^\prime\}$. The integral
$\Theta_{\mathcal{L}}$ in Definition~\ref{def:dists}
\begin{enumerate}
\item converges for every
  $f\in\mathcal{C}(G(K))$, therefore defining a distribution;
\item is an admissible distribution, as
  defined by Harish-Chandra~\cite[$\S 14$]{HC78}.
\end{enumerate}
\end{proposition}
\begin{proof} 
\begin{enumerate}
\item We will prove this by relating $\Theta_{\mathcal{L}}$ to integrals that
  are known to converge. To do this, it is necessary to first
  determine the trace of Frobenius functions attached to $\mathcal{E}$
  and $\mathcal{E}^\prime$ in the previous section. Recall that
\[
t^{\mathcal{E}}_{\operatorname{Fr}}(a) =
\operatorname{Trace} ((\varphi_{\cl{\mathcal{E}}})_a; \cl{\mathcal{E}}_a),\qquad
\forall a\in (\mathcal{U}_{0,k})(k).
\]
What is required, then, is to find the action of the composition of
isomorphisms in~\ref{eqn:stalkoperator} on the stalk $\cl{\mathcal{E}}_a$. 

We will do this
calculation for the points
\[
a = \left( \begin{array}{cc} 1 & 1 \\ 0 & 1 \end{array}\right) \ \ \ \
a^\prime = \left( \begin{array}{cc} 1 & \varepsilon \\ 0 &
    1 \end{array}\right),
\]
where $\varepsilon$ is a quadratic nonresidue in $\FF_q^\times$, since
$G$-equivariance together with the fact that all points in
$\mathcal{U}_{0,k}(k)$ are $G_k(k)$-conjugate to one or
the other of these points means that this is sufficient to know the trace
function completely.

The local system $\cl{\mathcal{E}}$ is defined by a character $\chi_s$ of
the \'etale cover
\[
\mathbb{A}^2_{0,\cl{k}} \stackrel{f_s}{\longrightarrow}
\mathcal{U}_{0,\cl{k}} .
\]
This cover is given on coordinates by
\begin{eqnarray*}
a & \longmapsto & 1 - xy \\
b & \longmapsto & x^2 \\
c & \longmapsto & -y^2 \\
d & \longmapsto & 1 +xy.
\end{eqnarray*}
Following Mars \& Springer~\cite[$\S 2$]{MarsSpringer}, the stalk
$\cl{\mathcal{E}}_{\cl{z}}$ at a $K$-rational point $\cl{z}$ can be
identified with 
\[
\{ \phi\colon f_s^{-1}(z)\rightarrow\cl{\QQ}_\ell\tq \lsup{\gamma}{\phi} =
\chi_s(\gamma )\varphi , \
\forall\gamma\in\algfg{\mathcal{U}_{0,\cl{k}},\cl{u}}\};
\]
\ie , the set of functions on the fibre of $f_s$ above $\cl{z}$ on which
the fundamental group acts (via precomposition with the action of the
fundmental group on elements of the fibre) according to the character that defines
$\cl{\mathcal{E}}$. Similarly, $(\Fr^*\cl{\mathcal{E}})_{\cl{z}}$ is
the set of functions on the image of $f_s^{-1}(z)$ under
$\Fr\colon(\mathbb{A}^2_{\cl{k}})_0(\cl{k}) \rightarrow
(\mathbb{A}^2_{\cl{k}})_0(\cl{k})$ satisfying the same restriction.

For the point $\cl{a}$, we have
\[
f_s^{-1}(a) = \{(1,0),(-1,0)\}.
\]
$\Fr$ acts trivially on these points, so the stalks
$\cl{\mathcal{E}}_{\cl{a}}$ and $(\Fr^*\cl{\mathcal{E}})_{\cl{a}}$ are
exactly the same, and the composition~\ref{eqn:stalkoperator} consists
of trivial isomorphisms. Therefore,
\[
t^{\mathcal{E}}_\Fr(a) = 1.
\]

On the other hand, $f_s^{-1}(a^\prime ) =
\{(\sqrt{\varepsilon},0),(-\sqrt{\varepsilon},0)\}$, on which $\Fr$ acts
nontrivially, swapping the two elements. The stalks
$\cl{\mathcal{E}}_{\cl{a}^\prime}$ and
$(\Fr^*\cl{\mathcal{E}})_{\cl{a}^\prime}$ are thus the same set of
functions, and the canonical isomorphism in the
composition~\ref{eqn:stalkoperator} is trivial, with trace 1. The
second isomorphism, however, has trace -1, owing to the fact that the
action of Frobenius on $f_s^{-1}(a^\prime )$ coincides with the action
of the nontrivial element
$\gamma\in\algfg{\mathcal{U}_{0,\cl{k}},\cl{u}}$ (the method used in
Chapter~\ref{ch:mainresult}, Section~\ref{ssection:fundamentalgroups}
to determine $\algfg{\cl{\mathcal{U}}_0,\cl{u}}$ works on the special
fibre as well, and yields an analogous result). If $\alpha^\prime\in
f_s^{-1}(a^\prime )$ and $\phi\in
(\Fr^*\cl{\mathcal{E}})_{\cl{a}^\prime}$,
\begin{align*}
(\varphi_{\cl{a}^\prime}(\phi ))(\alpha^\prime ) & = \phi
(\Fr^{-1}(\alpha^\prime ) \\ 
& = \phi (\gamma\cdot\alpha^\prime ) \\
& = \lsup{\gamma}{\phi}(\alpha^\prime ) \\
& = \chi_s(\gamma )\phi (\alpha^\prime ) \\
& = (-1)\phi (\alpha^\prime ).
\end{align*}
Therefore,
\[
t^{\mathcal{E}}_\Fr(a^\prime ) = -1.
\] 

Since all elements of $\mathcal{U}_{0,k}$ are
$\SL{2}{k}$-conjugate to either $a$ or $a^\prime$, and both
$\mathcal{E}$ and $\mathcal{E}^\prime$ are equivariant (this can be
proved analogously to Proposition~\ref{prop:equivariance}), these
calculations determine their trace functions. We justify this next.  

Since $\mathcal{E}$ and $\mathcal{E}^\prime$ are
$\SL{2}{k}$-equivariant, there is an isomorphism
\[
\Phi\colon c^*\mathcal{L} \iso p^*\mathcal{L}
\]
where $\mathcal{L}$ is either of those local systems, $c$ the
conjugation by $\SL{2}{k}$ morphism, and $p\colon G_k\times \mathcal{U}_{0,k}
= \SL{2}{k}\times \mathcal{U}_{0,k}$ the canonical projection. Let $g\in G_k(k)$
and define morphisms 
\[
c(g)\colon\mathcal{U}_{0,k}\longrightarrow\mathcal{U}_{0,k}
\]
and
\[
p(g)\colon\mathcal{U}_{0,k}\longrightarrow\mathcal{U}_{0,k}
\]
by the following commutative diagram.
\[
\xymatrix{
G_k\times \mathcal{U}_{0,k} \ar[rr]^{c\textrm{ (resp. $p$)}} & & \mathcal{U}_{0,k} \\
\Spec{k}\times \mathcal{U}_{0,k} \ar[u]_{g\times\id_{\mathcal{U}_{0,k}}} & & \\
\mathcal{U}_{0,k} \ar[u]^{\wr} \ar@/^5pc/[uu]^h \ar[uurr]_{c(g)\textrm{
    (resp. $p(g)$)}} & &}
\]
The object in the middle of the left side is simply the closed
subscheme $\{g\}\times \mathcal{U}_{0,k}$; the map $c(g)$ is just the
conjugation-by-$g$ map; and $p(g)$ is the composition of the inclusion
of $\mathcal{U}_{0,k}$ into the product with the canonical projection,
so in fact, for all $g$, $p(g) = \id_{\mathcal{U}_{0,k}}$. The
equivariance morphism $\Phi$ restricts to these morphisms, since
\begin{align*}
c(g)^*\mathcal{L} & =  (c\circ h)^*\mathcal{L} & \\
& =  h^*(c^*\mathcal{L})  & \\
& \iso  h^*(p^*\mathcal{L})  & \textrm{ (by }h^*\Phi )\\
& =  (p\circ h)^*\mathcal{L} & \\
& =  p(g)^*\mathcal{L}. &
\end{align*}
Restrict $\Phi$ to the closed subscheme $\{g\}\times
\mathcal{U}_{0,k}$. Composing the inclusion of that closed
subscheme with the isomorphism $\{g\}\times
\mathcal{U}_{0,k}\iso \mathcal{U}_{0,k}$ determines an
isomorphism
\[
\alpha (g)^*\mathcal{L} \iso p(g)^*\mathcal{L},
\]
Because $p(g) = \id_{\mathcal{U}_{0,k}}$, this isomorphism can be rewritten as 
\[
c(g)^*\mathcal{L} \iso \mathcal{L}.
\]
Arising as they do from a single isomorphism of schemes, the family of
these isomorphisms as $g$ varies over $G_k(k)$ are all
compatible. The trace of Frobenius function depends only on the
isomorphism class of the complex, so
\[
t_\Fr^{\mathcal{L}} = t_\Fr^{c(g)^*\mathcal{L}}.
\]
In addition (as in Laumon~\cite[1.1.1.4]{LaumonFT}), 
\[
t_\Fr^{c(g)^*\mathcal{L}} =
t_\Fr^{\mathcal{L}}\circ c(g).
\]
Thus, the calculations above entirely determine the trace function,
and
\[
t_\Fr^{\mathcal{L}}(x) = \left\{ \begin{array}{lc} 1 & x\textrm{ is in the
        orbit of }a, \\
-1 & x\textrm{ is in the orbit of }a^\prime ,  \\
0 & \textrm{ otherwise.} \end{array} \right.
\]

Finally, we can move to establishing the convergence of
$\Theta_{\mathcal{L}}$. For this, we refer to Table 2 in Chapter 15 of Digne and
Michel~\cite{dignemichel}. The fifth and sixth rows of this table
correspond to characters of irreducible cuspidal representations of
$G(k)$ that we will denote $\chi_+$ and $\chi_-$. Further inspection
of the table reveals that
\[
t_\Fr^{\mathcal{L}} = c(\chi_+ - \chi_-),
\]
where $c$ is a constant we can calculate but will safely ignore. 

Let $\sigma_+$ and
$\sigma_-$ be the irreducible cuspidal representations corresponding
to $\chi_+$ and $\chi_-$, and let $\pi_+$ and $\pi_-$ be depth-zero
supercuspidal representatons of $G(K)$ obtained from $\sigma_+$
and $\sigma_-$ via compact induction, inflated from the standard
maximal compact subgroup
(denoted $X(\mathcal{O}_K)$ in Chapter~\ref{ch:mainresult}) of $G(K)$, while
$\pi^\prime_+$ and $\pi^\prime_-$ denote the supercuspidals inflated
from the nonstandard maximal compact subgroup. These four depth-zero
supercuspidal representations are in fact the residents of the lone
quaternary L-packet for $G(K)$. Harish-Chandra~\cite[$(\S 5)$]{HC70} has described how
to associate characters (in fact, distributions on $G(K)$) to such
representations, which we denote $\Theta_{\pi_\pm}$ and
$\Theta_{\pi^\prime_\pm}$. These characters are given by the
following so-called Frobenius formula,
\begin{align*}
\Theta_{\pi_\pm}(f) =
\int_{G(K)}\int_{G(K)}f(ghg^{-1})f_\pm (h)dh,dg \\
\Theta_{\pi^\prime_\pm}(f) =
\int_{G(K)}\int_{G(K)}f(ghg^{-1})f^\prime_{\pm}(h)dh,dg
\end{align*}
(see, for example, Cunningham and Gordon~\cite{cunninggordonmotivic}). The
functions $f_\pm$ and $f^\prime_\pm$ appearing in the formul{\ae} are the locally
constant, compactly supported functions on $G(K)$ obtained by
inflating the characters $\chi_\pm$ to, respectively, the
standard and nonstandard maximal compact subgroup, and then extending
by zero (the same process that induced the function $f_{\mathcal{L}}$,
recall). These integrals are known to converge (by
Harish-Chandra~\cite[$\S 5$]{HC70}), so that the
differences
\begin{align*}
\Theta_{\pi_+}(f) - \Theta_{\pi_-}(f) =
\int_{G(K)}\int_{G(K)}f(ghg^{-1})f_{\oE}(h)dh,dg \\
\Theta_{\pi^\prime_+}(f) - \Theta_{\pi^\prime_-}(f) =
\int_{G(K)}\int_{G(K)} f(ghg^{-1})f_{\oE^\prime}(h)dh,dg
\end{align*}
likewise converge. These are the integrals $\Theta_{\mathcal{E}}$ and
$\Theta_{\mathcal{E}^\prime}$, respectively, and so they converge.
\item This is immediate from the definition, which includes linear
  combinations of distributions obtained as characters of admissible
  representations, in view of the
  fact that $\Theta_{\mathcal{L}}$ is equal to one of the differences of
  integrals above.
\end{enumerate}
\end{proof}

\section{Distributions on the Lie Algebra}

Further significant properties appear after transporting
$\Theta_{\mathcal{E}}$ and $\Theta_{\mathcal{E}^\prime}$ to the Lie
algebra of $G(K)$,
\[
\fg\ceq \mathfrak{sl}(2)_K= \Spec{K[a,b,c,d]/(a+d)}.
\]
To do so, we will make use of the modifed Cayley transform, which is
defined on points by
\begin{align*}
\operatorname{cay}\colon \fg (K) & \longrightarrow  G(K) \\
g & \longmapsto  \left( 1 + \frac{g}{2}\right)\left(1 -
  \frac{g}{2}\right)^{-1}.
\end{align*}
The transform gives an isomorphism between topologically nilpotent
elements of $\fg (K)$ and topologically unipotent elements in $G(K)$ (see
Cunningham \& Gordon~\cite{cunninggordonmotivic}). Precomposition by
$\operatorname{cay}$ gives an isomorphism of Hecke algebras that we
also call $\operatorname{cay}$:
\begin{align*}
\operatorname{cay}\colon\mathcal{C}(\fg (K)) &\longrightarrow 
\mathcal{C}(G(K)).
\end{align*}
\begin{definition}\label{def:distributions2}
For $\mathcal{L}\in\{\mathcal{E},\mathcal{E}^\prime\}$ and any
$f\in\mathcal{C}(\fg (K))$, let
\[
D_{\mathcal{L}}(f) \ceq \Theta_{\mathcal{L}}(\operatorname{cay}(f)).
\]
By Proposition~\ref{prop:properties1}, this defines distributions on
$\fg (K)$.
\end{definition}

Recall that the Fourier transform on $\fg (K)$ is defined by
first making a choice of character $\psi\colon K\rightarrow
\cl{\QQ}_\ell^\times$ and setting, for any $X,Y\in \fg (K)$, 
\[
\psi (X,Y) \ceq \psi (\operatorname{Trace}(XY)).
\]
Then the Fourier transform of any $f\colon \fg
(K)\rightarrow\cl{\QQ}_\ell^\times$ in $\mathcal{C}(\fg (K))$ is 
\[
\hat{f}(X) \ceq \int_{\fg (K)}f(Y)\psi (X,Y)dY.
\]
The Fourier transform of a distribution $D$ on $\fg (K)$ is then
\[
\hat{D}(f) \ceq D(\hat{f}), \quad f\in\mathcal{C}(\fg (K)).
\]

\begin{proposition}\label{prop:properties2}
Let $\hat{D}_\mathcal{L}$ be
the Fourier transform of $D_{\mathcal{L}}$ for $\mathcal{L}\in\{\oE ,\oE^\prime\}$. There exists a
$\lambda\in\cl{\QQ}_\ell$ such that $\hat{D}_{\oE} = \lambda D_{\oE}$
and $\hat{D}_{\oE^\prime} = \lambda D_{\oE^\prime}$. Moreover,
$D_{\oE}$ and $D_{\oE^\prime}$ are linearly independent in the space
of all admissible distributions on $\mathcal{C}(\fg (K))$.
\end{proposition}
\begin{proof}
\begin{enumerate}
\item 
For any $f\in\mathcal{C}(\fg (K))$,
\begin{eqnarray*}
\hat{D}_{\mathcal{L}}(f) & = & D_{\mathcal{L}}(\hat{f}) \\
& = & \int_{G(K)}\int_{\fg (K)}\hat{f}(y^{-1}Yy)\operatorname{cay}(f_{\mathcal{L}})(Y)dYdy \\
& = & \int_{G(K)}\int_{\fg (K)}f(y^{-1}Yy)\widehat{\operatorname{cay}(f_{\mathcal{L}})}(Y)dYdy .
\end{eqnarray*}
In Waldspurger~\cite[II.4,p.~34]{waldspurger}, the character $\chi_+ -
\chi_-$ appears, denoted by $\lsup{\circ}{f}$, and is identified as a
cuspidal function; by the first corollary in
Lusztig~\cite[\S~10]{lusztigfourier}, such functions are eigenfunctions of
the Fourier transform on $\fg (K)$. Thus, by Proposition 1.13 in Cunningham and
Hales~\cite{cunninghales}, 
\[
\hat{f}_{\mathcal{L}} = \lambda f_{\mathcal{L}}
\]
(The proposition gives the value of $\lambda$ explicitly). Therefore,
\[
\hat{D}_{\mathcal{L}} = \lambda D_{\mathcal{L}}.
\]
\item Follows from Table 12 in Cunningham and
  Gordon~\cite{cunninggordonmotivic}.
\end{enumerate}
\end{proof}


   %
   %

   %
   %
   \bibliography{sources}
   \bibliographystyle{amsplain}

   %
   %

   %
   %

\end{document}